\documentclass[onefignum,onetabnum]{siamart171218}

\usepackage{amsmath}
\usepackage{amsfonts}
\usepackage{amsopn}
\usepackage{cases}
\usepackage{booktabs}
\usepackage{graphicx}
\usepackage{enumerate}
\usepackage{xcolor}
\usepackage{subcaption}
\usepackage{multirow}

\newsiamremark{remark}{Remark}

\newcommand{\tr}{\mathrm{tr}}
\newcommand{\ob}{\mathcal{OB}}
\newcommand{\diag}{\mathrm{diag}}
\newcommand{\rank}{\mathrm{rank}}
\newcommand{\cgt}[1]{#1^{*}} 
\newcommand{\mi}{\imath}
\newcommand{\rep}{\mathrm{Re}}  
\newcommand{\imp}{\mathrm{Im}}  
\newcommand{\fnorm}[1]{\left\|#1\right\|_\mathrm{F}}  
\newcommand{\twonorm}[1]{\left\|#1\right\|_2}

\newcommand{\bbC}{\mathbb{C}}
\newcommand{\bbR}{\mathbb{R}}

\newcommand{\bra}[1]{\langle #1|}
\newcommand{\ket}[1]{|#1 \rangle}
\newcommand{\veps}{\varepsilon}

\newcommand{\hH}{\hat{H}}
\newcommand{\hU}{\hat{U}}

\newcommand{\tX}{\widetilde{X}}

\newcommand{\tU}{\widetilde{U}}
\newcommand{\tA}{\widetilde{A}}
\newcommand{\tB}{\widetilde{B}}
\newcommand{\tH}{\widetilde{H}}
\newcommand{\tV}{\widetilde{V}}
\newcommand{\tLambda}{\widetilde{\Lambda}}
\newcommand{\tSigma}{\widetilde{\Sigma}}
\newcommand{\ta}{\tilde{a}}
\newcommand{\tlambda}{\tilde{\lambda}}
\newcommand{\tx}{\tilde{x}}
\newcommand{\tsigma}{\tilde{\sigma}}

\newcommand{\bU}{\bar{U}}
\newcommand{\bV}{\bar{V}}
\newcommand{\bSigma}{\bar{\Sigma}}
\newcommand{\blambda}{\bar{\lambda}}

\newcommand{\moleculeHtwo}{\mathrm{H}_2}

\title{Landscape Analysis of Excited States Calculation over Quantum
Computers}

\headers{Landscape Analysis of Excited States VQE}{Hengzhun Chen, Yingzhou Li,
    Bichen Lu, Jianfeng Lu}

\author{Hengzhun Chen\thanks{School of Mathematical Sciences, Fudan
    University (\email{hengzhunchen21@m.fudan.edu.cn}).} 
\and Yingzhou Li\thanks{School of Mathematical Sciences, Shanghai Key
    Laboratory for Contemporary Applied Mathematics, Fudan University and
    Key Laboratory of Computational Physical Sciences, Ministry of Education
    (\email{yingzhouli@fudan.edu.cn}).} 
\and Bichen Lu\thanks{Shanghai Center for Mathematical Sciences, Fudan
    University (\email{bclu18@fudan.edu.cn}).} 
\and Jianfeng Lu\thanks{Department of Mathematics, Department of Physics
    and Department of Chemistry, Duke University
    (\email{jianfeng@math.duke.edu}).}}

\begin{document}

\maketitle

\begin{abstract}
    The variational quantum eigensolver (VQE) is one of the most promising
    algorithms for low-lying eigenstates calculation on Noisy
    Intermediate-Scale Quantum (NISQ) computers. Specifically, VQE has
    achieved great success for ground state calculations of a Hamiltonian.
    However, excited state calculations arising in quantum chemistry and
    condensed matter often requires solving more challenging problems than
    the ground state as these states are generally further away from a
    mean-field description, and involve less straightforward optimization
    to avoid the variational collapse to the ground state. Maintaining
    orthogonality between low-lying eigenstates is a key algorithmic
    hurdle. In this work, we analyze three VQE models that embed
    orthogonality constraints through specially designed cost functions,
    avoiding the need for external enforcement of orthogonality between
    states. Notably, these formulations possess the desirable property
    that any local minimum is also a global minimum, helping address
    optimization difficulties. We conduct rigorous landscape analyses of
    the models' stationary points and local minimizers, theoretically
    guaranteeing their favorable properties and providing analytical tools
    applicable to broader VQE methods. A comprehensive comparison between
    the three models is also provided, considering their quantum resource
    requirements and classical optimization complexity.
\end{abstract}

\begin{keywords}
    Variational quantum eigensolver, excited states calculation,
    landscape analysis, oblique manifold
\end{keywords}

\section{Introduction}

With the availability of near-term quantum devices and significant
advancements in quantum supremacy experiments \cite{arute2019quantum,
    zhong2020quantum}, quantum computing has gained substantial attention from
various scientific disciplines in recent years. Algorithms designed for
these restricted devices, known as noisy intermediate-scale quantum (NISQ)
devices \cite{preskill2018quantum}, have distinct characteristics: they
operate on a small number of qubits, exhibit some level of noise
resilience, and often employ hybrid approaches that combine quantum and
classical computing steps. An example is the variational quantum
eigensolver (VQE) \cite{mcclean2016theory, peruzzo2014variational}, which
shows particular promise within the NISQ paradigm.

In general, VQE aims to variationally determine an upper bound to the
ground state energy of a Hamiltonian, which enables important energetic
predictions for molecules and materials \cite{deglmann2015application}.
Namely, given a Hamiltonian $\hH$ with any normalized trial wavefunction
$\ket{\psi}$, the ground state energy $E_0$ associated with $\hH$ is
bounded by $E_0 \leq \bra{\psi}\hH\ket{\psi}$. The goal of VQE is to find
a parametrization of $\ket{\psi}$ such that the expectation value of $\hH$
is minimized. Note that the normalized trial wavefunction $\ket{\psi}$
over a quantum computer is implemented by $\hU(\theta)\ket{0}$, where
$\hU(\theta)$ is a parameterized unitary operation. We can rewrite the VQE
optimization problem with corresponding cost function as
\begin{equation*}
    E_{\text{VQE}} = \min_{\theta} ~ \bra{0} \hU^{\dagger}(\theta) \hH \hU(\theta) \ket{0}.
\end{equation*}
The hybrid nature of VQE arises from the fact that the expectation value can be
computed on a quantum device while the minimization problem with respect
to $\theta$ is computed on a classical computer.

Compared to other quantum eigensolvers, such as quantum phase estimation
(QPE) \cite{abrams1999quantum,aspuru2005simulated,kitaev1995quantum},
which require circuit depths far beyond the capabilities of near-term
intermediate-scale quantum (NISQ) devices, the variational quantum
eigensolver (VQE) offers a more NISQ-friendly approach. VQE relies on
shorter circuit depths and fewer qubits but comes at the cost of
increased number of measurements, parameter optimization on a classical
computer and the need for a predefined ansatz to approximate
eigenstates. The choice of an appropriate ansatz is a critical component
of the VQE workflow and has been the focus of significant research in
recent years.

The hardware-efficient ansatz (HEA) \cite{kandala2017hardware} directly
parameterizes quantum states using native quantum gates, making it
well-suited for certain problem-specific scenarios. However, it is often
inadequate for large-scale chemical problems due to its limited
scalability. In contrast, the unitary coupled cluster (UCC) ansatz
\cite{peruzzo2014variational}, along with its extensions such as UCCSD
\cite{romero2018strategies}, UCCGSD \cite{lee2018generalized}, and
k-UpCCGSD \cite{lee2018generalized}, has demonstrated strong numerical
accuracy with only linear scaling in quantum resource requirements. This
makes UCC-based ansatzes a promising trade-off between computational
cost and accuracy. Recent studies
\cite{grimsley2019adaptive,ryabinkin2018qubit} have also explored
adaptive ansatz structures, which allow the ansatz dimensions to vary
dynamically. These approaches aim to provide reliable alternatives or
improvements over fixed-structure ansatzes, further enhancing VQE's
flexibility and efficiency.

A promising application for VQE is ab initio electronic structure
calculation. In particular, VQE solves the many-body problem under the
full configuration interaction (FCI) framework \cite{blunt2015excited,
    holmes2017excited, li2020optimal, schriber2017adaptive,
    wang2019coordinate}, where the time-independent Schr\"odinger equation is
discretized exactly and the ground state is computed by solving the
standard eigenvalue problem exponential-scaling with the number of
electrons in the system. The exponential cost of classical FCI fits well
with capabilities of quantum algorithms. Accordingly, quantum algorithms
for ground state calculation under the FCI framework have been extensively
explored over recent decades \cite{abrams1999quantum, aspuru2005simulated,
    bauer2020quantum, mcardle2020quantum, peruzzo2014variational}. Among them,
VQE is a quantum-classical hybrid approach with the shortest circuit
depth, which makes it the most widely applied quantum eigensolver on NISQ
devices.

In addition to the ground state, the consideration of excited states is
also crucial in quantum chemistry. However, computing excited states is
generally a more challenging task, due to the excited states generally
being further away from a mean-field description, as well as less
straightforward optimization to avoid the variational collapse to the
ground state. Existing quantum algorithms for excited states computation
can be broadly divided into two main types of methods: the
perturbation-based approaches and the variational approaches.

Perturbation-based methods always start from the ground state via VQE and
then evaluate the expansion values on single excited states from the
ground state. With these expansion values they further utilize different
approaches to obtain the low-lying excited state energies. The
representatives in this category could be the quantum subspace expansion
(QSE) method \cite{colless2018computation, ganzhorn2019gate,
    mcclean2017hybrid} and quantum equation of motion (qEoM)
\cite{ollitrault2020quantum}.

For the variational approach, it uses modified VQE cost functions to afford a
fully variational flexibility and maintain the orthogonality of the
low-lying eigenstates. These methods have the advantage of avoidance for
the limitations and biases of perturbation-based methods, but usually come
with a higher cost of quantum resources and the restriction to a specific
ansatz chosen. For instance, the folded spectrum method
\cite{mcclean2016theory} folds the spectrum of a Hamiltonian operator via
a shift-and-square operation, i.e., $\hH \rightarrow (\hH-\lambda I)^2$,
where $\lambda$ is a chosen number and as such the ground state of
$(\hH-\lambda I)^2$ is the eigenstate of $\hH$ whose eigenvalue is closest to
$\lambda$. Then $(\hH-\lambda I)^2$ is used as the Hamiltonian operator in
VQE to obtain the desired excited states. The witness-assisted variational
eigenspectra solver \cite{santagati2017finding,santagati2018witnessing}
incorporates an entropy term in the objective function to prevent
variational collapse back to the ground state by requiring the minimized
excited state to be close to its initial state, which is defined as an
excitation from the ground state. The quantum deflation method
\cite{higgott2019variational,jones2019variational} adopts the overlapping
term between the target state and the ground state as a penalty term in
the objective function to keep the target state as orthogonal as possible
to the ground state.

While the methods above all have their classical algorithm counterpart, two
methods, subspace search VQE (SSVQE) \cite{nakanishi2019subspace} and
multistate contracted VQE (MCVQE) \cite{parrish2019quantum} leverage a unique
property of quantum computing to maintain the orthogonality between
eigenstates: all quantum operations are unitary and a unitary
transformation will not change the orthogonality of the states it is
applied to. They first prepare $K$ non-parameterized mutually orthogonal
states $\ket{\phi_1}, \ket{\phi_2}, \ldots, \ket{\phi_K}$ and then
apply a parameterized unitary circuit $\hU_{\theta}$ to these states.
Hence, we can minimize the objective function
\begin{equation*}
    \sum_{k=1}^K \bra{\phi_k} \hU_{\theta}^{\dagger} \hH \hU_{\theta} \ket{\phi_k},
\end{equation*}
to obtain the desired eigenstates simultaneously. However, energies
associated with $\hU_{\theta}\ket{\phi_1}$, $\ldots$,
$\hU_{\theta}\ket{\phi_K}$ may not be ordered. Hence, SSVQE introduces a
min-max problem to obtain a specific excited state and a weighted
objective function to preserve the ordering of eigenstates in some further
extensions. Instead, MCVQE uses a classical post-processing to extract the
ordered eigenstates from the eigen-subspaces obtained from the
optimization procedure. Specifically, it gives the approximated
eigenstates in the form of $\ket{\psi_k}=\sum_l \hU_{\theta} \ket{\phi_l}
    V_{lk}$, where $V_{lk}$ is the $(l,k)$-th entry of a unitary matrix
$V$, which is the eigenvector matrix of the entangled contracted
Hamiltonian $H_{kl}:= \bra{\phi_k} \hU^\dagger_{\theta} \hH \hU_{\theta}
    \ket{\phi_l}$.

As discussed, maintaining orthogonality between low-lying eigenstates is a
central algorithmic challenge for excited state computations, regardless
of whether orthogonality constraints are addressed explicitly or
implicitly. Motivated by the quantum orbital minimization method (qOMM)
\cite{bierman2022quantum}, in this work we will theoretically analyze
several approaches that embed the orthogonality constraints into the
cost function itself, resulting in an ``unconstrained optimization" that fits
better to the quantum computer. Specifically, we obtain the energy
landscape of these methods (including the qOMM), which provides
theoretical guarantees for the favorable properties of the corresponding
quantum algorithms.

Now we give the mathematical formulation of the excited states quantum
eigensolver problem in matrix notation for later discussion. Given
Hermitian matrix $A \in \bbC^{n\times n}$, our target is to find the
smallest $p$ eigenpairs of $A$, denoted as
\begin{equation*}
    A X = X \Lambda, \quad X \in \bbC^{n\times p},
\end{equation*}
where $\Lambda$ is the diagonal eigenvalue matrix. In the quantum
computing setting, each column of $X$ is represented by ansatz like
$\ket{\psi_i}=\hU_{\theta_i} \ket{\phi_i}$ with $\hU_{\theta_i}$ being a
unitary operator for $i=1, \ldots, p$. If quantum error is not taken into
account, each column of $X$ will always be of unit length, which gives the
constraints for diagonal elements of $\cgt{X}X$. Further, this constraint
corresponds to the oblique manifold, which is defined as ~\footnote{For
    the sake of notation, we adopt $\diag(\cdot)$ similar to the MATLAB
    ``diag'' function, i.e., $\diag(v)$ is a square diagonal matrix with the
    entries of vector $v$ on the diagonal and $\diag(A)$ is a column vector of
    the diagonal entries of $A$.}
\begin{equation} \label{def:obli_mani}
    \ob(n,p) = \left\{ X \in \bbC^{n\times p} : \text{diag}(X^* X) = \mathbf{1} \right\},
\end{equation}
where $\mathbf{1}$ denotes an all-one vector of length $p$. Thus, we
emphasize that although we apply the unconstrained optimization problem
over a quantum computer, the quantum computer itself puts the extra oblique
manifold constraints to the optimization problem due to the normality
condition of any quantum state. Moreover, we remark that since the
representability of ansatz circuit is typically sufficient, i.e., the
variational space is rich enough, especially for the chemical systems, we
ignore the ansatz design and investigate the energy landscape of VQE cost
function with respect to $X \in \ob(n, p)$ directly.

Consider the well known optimization formulation of classical eigensolver,
~\footnote{We use $I_p$ to denote the identity matrix of size $p$-by-$p$.
    For notation simplicity, we also use $I$ to denote the identity matrix
    with size consistent to the matrix operations.}
\begin{equation*}
    \min_{X\in \bbC^{n\times p}} \tr(\cgt{X}AX), \quad \text{s.t.} \quad \cgt{X}X=I_p.
\end{equation*}
To address the orthogonal constraints for excited states computation, we
introduce the three models below as examples; each provides a method to
embed the orthogonality constraint into the objective function, while
the oblique manifold constraint arises from the nature of quantum
computers,
\begin{align}
     & \min_{X\in \ob(n,p)} \tr \left( (2I - X^*X) X^*AX\right), \tag{qOMM}                                \\
     & \min_{X\in \ob(n,p)} \frac{1}{2}\tr(\cgt{X}AX) + \frac{\mu}{4} \fnorm{\cgt{X}X- I}^2, \tag{qTPM}    \\
     & \min_{X\in \ob(n,p)} \tr(\cgt{X}AX) + \mu_ 1 \sum_{i< j} \left| (\cgt{X}X)_{ij} \right|, \tag{qL1M}
\end{align}
where $\fnorm{\cdot}$ denotes the Frobenius norm and $\mu, \mu_1 > 0$
are finite penalty constants. Each objective function above leads to a
quantum eigensolver. Additionally, qOMM and qTPM require matrix $A$ to
be negative definite Hermitian, while qL1M works for any
Hermitian matrix $A$ and requires $p < n$. We note that the first model
has been developed as a hybrid quantum-classical algorithm, termed the
quantum orbital minimization method (qOMM) \cite{bierman2022quantum}.
The latter two models (qTPM and qL1M) have their own favorable properties
and are promising candidates for VQE algorithms. In this work, we analyze the
energy landscape of the proposed models. As the main contribution, we
present theorems that characterize the stationary points and local
minimizers of all three models. Notably, all three models possess the
favorable property that any local minimum is also a global minimum.

The rest of the paper is organized as follows. In \cref{sec:landscape},
all three models will be investigated in detail. A formal statement,
together with their landscape analysis and algorithm frameworks are
given. The complete proof of the theorems will be presented in
the appendix. In \cref{sec:numerical}, we present simulated quantum
numerical results to demonstrate the convergence of the VQE algorithms.
Furthermore, we compare their efficiency in terms of quantum resource
requirements and classical optimization complexity. Finally,
\cref{sec:conclusion} concludes the paper and discusses future work.

\section{Optimization Landscape and Algorithms in VQE} \label{sec:landscape}

In this section, we provide a detailed landscape analysis and the
algorithm framework of the three proposed formulations. Proofs of the
theorems stated here are deferred to the appendix. Specifically,
we characterize and contrast properties of stationary points and local
minimizers for these formulations. This lays theoretical groundwork for
understanding their algorithmic behavior and deriving optimization
algorithms.

\subsection{qOMM}

The original orbital minimization method (OMM) was initially developed
to solve for the low-lying eigenspace of a negative definite Hermitian
matrix $A$ in Kohn-Sham density functional theory
\cite{mauri1993orbital, ordejon1993unconstrained}. It employs the energy
functional given by
\begin{equation*}
    E_0(X) = \tr\left( (2I - \cgt{X} X ) \cgt{X} A X \right),
\end{equation*}
where $X$ represents a matrix of orbitals. Notably, OMM possesses a few
favorable properties \cite{LU201787}, which provides some benefits for
the algorithm design. First, even though the orthogonality constraint
among the columns of $X$ is not explicitly enforced, the minimizers of
$E_0(X)$ inherently consist of mutually orthogonal vectors. Second,
although the energy functional is non-convex, each local minimum is also
a global minimum.

OMM has been generalized over quantum computer as a VQE, and it leads to
the qOMM method, where the energy landscape is unknown but numerically
performs well. Correspondingly, we aim to characterize the stationary points
and local minimizers of qOMM, which are given by the optimization problem:
\begin{equation} \label{eq:qomm}
    \min_{X\in \ob(n,p)} E_0(X)=\tr \left( (2I - \cgt{X}X) \cgt{X}AX\right).
\end{equation}
In contrast to the straightforward characterization of stationary points
for the unconstrained optimization in OMM, \cref{thm:stationary_qomm}
reveals a more intricate structure of the stationary points for qOMM. This
subtle result for stationary points can then be leveraged to deduce the
desired property of local minimizers stated in \cref{thm:local_min_qomm}.

\begin{theorem} \label{thm:stationary_qomm}
    Given a negative definite Hermitian matrix $A$, stationary
    points of qOMM \eqref{eq:qomm} take the form $XP$, where $P \in
        \bbR^{p \times p}$ is an arbitrary permutation matrix, $X \in
        \bbC^{n \times p}$ admits a block structure, $X = \left( X_1,
        \cdots, X_k \right)$, and $k$ is the number of blocks. For each
    block $X_i \in \bbC^{n \times p_i}$, we denote the reduced SVD as
    $X_i = U_i \Sigma_i V_i^*$ and the rank as $r_i$, where $\Sigma_i =
        \diag(\sigma_{i1}, \dots, \sigma_{ir_i})$ is the nonzero singular
    value matrix. Matrix $U$ is composed of column bases of $X_i$s,
    i.e., $U = \left( U_1, \cdots, U_k \right)$. The
    following conditions hold.
    \begin{enumerate}
        \item The matrix $U$ is partial unitary and diagonalize $A$,
              i.e., $U^*U = I$ and $U^* A U$ is diagonal.
        \item Denote the $j$-th diagonal entry of $U_i^* A U_i$ as
              $\ta_{ij}$. Singular values are,
              \begin{equation*}
                  \sigma_{ij} =
                  \sqrt{1 + \frac{p_i - r_i}{\ta_{ij} \cdot
                          \sum_{\ell=1}^{r_i}\frac{1}{\ta_{i \ell}}}} \geq 1.
              \end{equation*}
              Furthermore, if the $j$-th column $u_{ij}$ of $U_i$ is not an
              eigenvector of $A$, then $\sigma_{ij}=\sqrt{2}$.
        \item At most one block $X_i$ is of full rank, and then
              $X_i = U_i V_i^*$;
        \item The right singular vector matrix $V_i \in \bbC^{p_i\times
                      r_i}$ is partial unitary such that $X_i \in \ob(n, p_i)$. The
              existence of $V_i$ is guaranteed for any given $U_i$.
    \end{enumerate}
    In addition, any $X\in \bbC^{n\times p}$ satisfying the conditions above
    is a stationary point of qOMM \eqref{eq:qomm}.
\end{theorem}

The proof of \cref{thm:stationary_qomm} is given in \cref{sec:proof_qomm}.

\begin{remark}
    When $X_i$ is full rank, all the columns of $U_i$ are eigenvectors
    of matrix $A$. Otherwise $U_i$ may not be eigenvectors. This is very
    different from OMM without constraint. Here we give a concrete
    example for illustration. Suppose $A$ is a $7$-by-$7$ negative
    definite Hermitian matrix with eigendecomposition $A=Q\Lambda Q^*$
    where
    \begin{equation*}
        Q=\left(q_1, \cdots, q_7\right), \quad \Lambda=\diag(-1, -2, -3, -4, -5, -6, -7).
    \end{equation*}
    Now we construct a stationary point of size $7$-by-$5$ with rank
    $2$ from its reduced SVD,
    \begin{equation*}
        X = U\Sigma V^* := \left(q_2, u_2\right) \begin{pmatrix}
            \sqrt{3} & \\  & \sqrt{2}
        \end{pmatrix} V^*,
        \quad u_2=\sqrt{\frac{2}{5}} q_1 + \sqrt{\frac{3}{5}} q_6,
    \end{equation*}
    where $V\in \bbC^{2\times 5}$ is partial unitary such that
    $\diag(X^*X)=\mathbf{1}$. Let $D=-8I$, substitute $X$ into the zero
    gradient condition \eqref{eq:Lagrange_qomm_zero_grad}, one could
    verify that
    \begin{equation*}
        2AX - AXX^*X - XX^*AX + XD = 0,
    \end{equation*}
    which shows that such $X$ is a stationary point but the second
    column of $U$ is not an eigenvector of $A$.
\end{remark}

\begin{theorem} \label{thm:local_min_qomm}
    Given a negative definite Hermitian matrix $A$, the local minimizers
    of qOMM \eqref{eq:qomm} take the form of $X=Q_p V^*$, where $Q_p$
    are eigenvectors of $A$ corresponding to the $p$ smallest
    eigenvalues, $V \in \bbC^{p \times p}$ is an arbitrary unitary
    matrix. Conversely, any matrix in the form of $Q_p V^*$ is a local
    minimizer of qOMM. Furthermore, any local minimum of qOMM is also a
    global minimum.
\end{theorem}

The proof of \cref{thm:local_min_qomm} is given in \cref{sec:proof_qomm}.

\begin{remark}
    Let matrix $A$ have eigenvalues ordered non-decreasingly as
    $\lambda_1 \leq \lambda_2 \leq \cdots \leq \lambda_n$. The
    eigenvectors corresponding to the smallest $p$ eigenvalues of $A$
    may not be unique if $\lambda_p = \lambda_{p+1}$, leading to
    degeneracy. In this paper, when we refer to the ``smallest $p$
    eigenvalues", we include any valid choice of eigenvectors in cases
    of degeneracy.
\end{remark}

From the landscape analysis above, we can observe that all the local
minima have the same objective function value and constitute the set of
global minima. The minimizers at these local minima consist of mutually
orthogonal vectors. This property leads to an algorithm design where
explicit orthogonal constraints are not required. Rather, orthogonality of
the minimizers can naturally emerge throughout the optimization process
without any orthogonalization steps, which allows for greater flexibility
in the choice of ansatz applied to each input state and the choice for
input states itself.

Given that the algorithm design and numerical experiments have been
thoroughly discussed in \cite{bierman2022quantum}, here we omit the
details of the algorithm of qOMM. Note that the algorithm design of qOMM
and qTPM are similar; one can also refer to \cref{sec:qtpm} for the framework.

\subsection{qTPM} \label{sec:qtpm}

It is well known that the invariant subspace associated with the $p$
algebraically smallest eigenvalues of matrix $A$ yields an optimal
solution to the trace minimization problem with orthogonality constraints
\begin{equation*}
    \min_{X\in \bbC^{n\times p}} \tr(\cgt{X}AX), \quad \text{s.t.} \quad \cgt{X}X=I.
\end{equation*}
Thus, we can introduce a penalty term using the Frobenius norm to
construct the corresponding unconstrained optimization problem:
\begin{equation} \label{eq:fro_trace_penalty}
    \min_{X\in \bbC^{n\times p}} f_{\mu}(X) :=
    \frac{1}{2} \tr(\cgt{X}AX) + \frac{\mu}{4}\fnorm{\cgt{X}X - I}^2,
\end{equation}
where the penalty parameter $\mu > 0$ takes suitable finite values. This
trace penalty minimization method (TPM) finds the optimal eigenspace and
excludes all full-rank non-optimal stationary points for an
appropriately chosen $\mu$ \cite{gao2024weighted,wen2016trace}. This
differs from general classical quadratic penalty methods, which require
the penalty parameter $\mu$ to go to infinity for the penalty model to
approach the original constrained model.

Now consider applying the TPM model to the quantum computing setting, which gives the
quantum trace penalty minimization method (qTPM), i.e.,
\begin{equation} \label{eq:qtpm_original}
    \min_{X\in \ob(n,p)} f_{\mu}(X) = \frac{1}{2} \tr(\cgt{X}AX) +
    \frac{\mu}{4}\fnorm{\cgt{X}X - I}^2.
\end{equation}
Note that for $X$ in oblique manifold $\mathcal{OB}(n,p)$, one has
\begin{equation*}
    \fnorm{\cgt{X}X - I}^2= \tr(\cgt{X}X\cgt{X}X) + \tr(I - 2\cgt{X}X) = \tr(\cgt{X}X\cgt{X}X) - p.
\end{equation*}
Thus, the optimization model \eqref{eq:qtpm_original} is equivalent to
\begin{equation} \label{eq:qtpm}
    \min_{X\in \ob(n,p)} g_{\mu}(X) := \frac{1}{2} \tr(\cgt{X}AX) +
    \frac{\mu}{4} \tr(\cgt{X}X\cgt{X}X).
\end{equation}
This formulation gives some flexibility and convenience for theoretical
analysis. Consequently, we also refer to it as the qTPM model.

\begin{remark}
    With $\mu >0$ the TPM model \eqref{eq:fro_trace_penalty} is equivalent
    to another useful method for large-scale eigenspace computation called
    the symmetric low-rank product (SLRP) model \cite{li2019coord,
        liu2015efficient} defined as
    \begin{equation*}
        \min_{X\in \bbC^{n\times p}} \frac{1}{2} \fnorm{X\cgt{X}- A}^2,
    \end{equation*}
    in the sense that any stationary point of one problem, after scaling
    and shifting, has a one-to-one correspondence with a stationary point
    of the other. This equivalence comes from the identity
    \begin{equation*}
        \frac{1}{2} \tr(\cgt{X}AX) + \frac{\mu}{4}\fnorm{\cgt{X}X - I}^2 \equiv
        \frac{\mu}{4}\fnorm{X\cgt{X} - (I - \frac{A}{\mu})}^2 + \text{constant}.
    \end{equation*}
    Since these two models are equivalent, we only discuss the qTPM,
    the generalization of SLRP model to the quantum computing setting can be carried
    out in the same manner.
\end{remark}

Now let us analyze the energy landscape of qTPM. While the local
minimizers of OMM and qOMM are the same, the explicit penalty term in
Frobenius norm leads to a different behavior for TPM and qTPM. Note that
local minimizers of TPM are scaled eigenvectors of matrix $A$, right
multiplying any unitary matrix \cite{wen2016trace}, which does not
satisfy the constraint in qTPM. Hence the local minimizers of qTPM are
different from that in TPM. Specifically, we have the following results
describing the landscape properties of qTPM.

\begin{theorem} \label{thm:stationary_qtpm}
    Given $\mu > |\lambda_{\min}(A)|$ and a negative definite Hermitian
    matrix $A$ \footnote{We use $\lambda_{\min}(\cdot)$ and
        $\lambda_{\max}(\cdot)$ to denote the minimum and maximum
        eigenvalues of a matrix respectively.}, stationary points of qTPM
    \eqref{eq:qtpm} take the form $XP$, where $P$ is an arbitrary permutation
    matrix and $X = \left( X_1, \cdots, X_k \right)$, where $X_i\in
        \bbC^{n\times p_i}$ with rank $r_i$ for $i=1, \ldots, k$. Denote the
    singular value decomposition of each $X_i$ as $X_i=U_i \Sigma_i
        V_i^*$, where $U_i\in\bbC^{n\times p_i}$, $\Sigma_i\in
        \bbR^{p_i\times p_i}, V_i \in \bbC^{p_i\times p_i}$. Let
    $$
        U=\left(U_1, \cdots, U_k \right), \quad
        \Sigma=\diag\left(\Sigma_1, \cdots, \Sigma_k\right), \quad
        V=\diag\left(V_1, \cdots, V_k\right),
    $$
    we have $X=U\Sigma V^*$, such $X$ satisfies the following
    conditions:
    \begin{enumerate}
        \item $U^*U=I$, each $U_i$ consists of eigenvectors of $A$ with
              eigenvalues $\lambda_{i1}, \ldots, \lambda_{ip_i}$ for $i=1,
                  \ldots, k$.
        \item For $i=1, \ldots, k$, $\Sigma_{i}=\diag\left(\sigma_{i1},
                  \cdots, \sigma_{ip_i}\right)$ with
              $$
                  \sigma_{ij} = \sqrt{\frac{p_i}{r_i} - \frac{1}{\mu} \left( \lambda_{ij} -\frac{1}{r_i} \sum_{l=1}^{r_i} \lambda_{il} \right)},
              $$
              for $j=1, \ldots, r_i$; $\sigma_{ij}=0$, $j=r_i+1, \ldots, p_i$.
        \item $V^*V=I$, each $V_i \in \bbC^{p_i\times p_i}$ is unitary
              such that $X_i\in \ob(n, p_i)$ for $i=1, \ldots, k$.
    \end{enumerate}
    Besides, any $X\in \bbC^{n\times p}$ satisfying the conditions above
    is a stationary point of qTPM \eqref{eq:qtpm}.
\end{theorem}

\begin{theorem} \label{thm:local_min_qtpm}
    Local minimizers of qTPM \eqref{eq:qtpm} with $\mu >
        |\lambda_{\min}(A)|$ and negative definite Hermitian matrix $A$ take
    the form of
    \begin{equation} \label{eq:local_min_qtpm}
        X=Q_p (I - \frac{1}{\mu}(\Lambda_p -\bar{\Lambda}_p))^{1/2} V^*,
        \quad \bar{\Lambda}_p = \frac{1}{p} \tr(\Lambda_p) I,
    \end{equation}
    where $Q_p$ are eigenvectors of $A$ corresponding to the $p$ smallest
    eigenvalues $\Lambda_p$ (counting geometric multiplicity),
    $V \in \bbC^{p \times p}$ is unitary matrix such that
    $\diag(X^*X)=\mathbf{1}$. Conversely, any matrix in the form
    \eqref{eq:local_min_qtpm} is a local minimizer of qTPM. Furthermore,
    any local minimum of qTPM is also a global minimum.
\end{theorem}

The proofs of \cref{thm:stationary_qtpm} and \cref{thm:local_min_qtpm}
are given in \cref{sec:proof_qtpm}.

From the landscape analysis above, we can also observe that all the local
minima have the same objective function value and constitute the set of
global minima. Although the local minimizers of qTPM have a more intricate
structure than qOMM, they nonetheless encode the target eigen-subspace
associated with the $p$ smallest eigenvalues of matrix $A$ (counting
for geometric multiplicity). Through a post-processing procedure such as
Rayleigh-Ritz projection, the final eigenpair approximations can be
extracted. Thus, qTPM successfully captures the pertinent
eigen-information at its minimizers, even if not in the simplest form
prior to post-processing.

We now present an algorithm framework for qTPM. In the quantum
computing setting of qTPM, each column of $X$ is represented by the ansatz
$\ket{\psi_i}=\hU_{\theta_i} \ket{\phi_i}$ for $i=1, \ldots, p$, where
$\ket{\phi_1}, \ldots, \ket{\phi_p}$ are initial states, $\hU_{\theta_1},
    \ldots, \hU_{\theta_p}$ are parameterized circuits with $\theta_1, \ldots,
    \theta_p$ being parameters. Thus, our objective function and optimization
problem \eqref{eq:qtpm} turn into
\begin{equation*}
    \min_{\theta_1, \cdots, \theta_p} \quad
    \frac{1}{2} \sum_{i=1}^p \bra{\phi_i} \hU^\dagger_{\theta_i} \hH \hU_{\theta_i} \ket{\phi_i} + \frac{\mu}{2}
    \sum_{i,j=1, i<j}^p \left|\bra{\phi_i} \hU^\dagger_{\theta_i} \hU_{\theta_j} \ket{\phi_j} \right|^2,
\end{equation*}
where $\hH$ is the Hamiltonian operator of the system. Given this
variational formulation, one can construct quantum circuits to evaluate
both inner products appearing in the objective function above, i.e.,
$\bra{\psi_i}\hH\ket{\psi_j}$ and $\langle \psi_i | \psi_j \rangle$
respectively. When the evaluation of the objective function is available,
we can adopt a gradient-free optimizer to minimize the objective function
with respect to parameters $\theta_1, \ldots, \theta_p$. After the
optimization process, we obtain a set of parameters such that
$\hU_{\theta_1}\ket{\phi_1}, \ldots, \hU_{\theta_p}\ket{\phi_p}$
approximate the target eigenspace $Q_pS$, where $Q_p\in \bbC^{n\times p}$ are
$p$ eigenvectors of $\hH$ associated with the smallest $p$ eigenvalues and
$S\in \bbC^{p\times p}$ is an invertible matrix. Thus we can
perform a Rayleigh-Ritz step to extract the desired eigenspace.
Specifically, we construct two matrices, $B$ and $C$, with their
$(i,j)$-th elements defined as follows:
\begin{equation*}
    B_{ij} = \bra{\phi_i} \hU^\dagger_{\theta_i} \hH \hU_{\theta_j} \ket{\phi_j}
    \quad \text{and} \quad
    C_{ij} = \bra{\phi_i} \hU^\dagger_{\theta_i} \hU_{\theta_j} \ket{\phi_j}.
\end{equation*}
We then solve a generalized eigenvalue problem with the matrix pencil $(B,C)$,
given by:
\begin{equation*}
    BR = CR\Lambda,
\end{equation*}
where $R$ represents the eigenvector matrix, and $\Lambda$ is a diagonal
matrix with the eigenvalues of $(B, C)$ as its diagonal entries.
Finally, we obtain the desired eigenvalues contained in $\Lambda$, along
with the corresponding eigenvectors $\sum_{j=1}^p
    \hU_{\theta_j}\ket{\phi_j}R_{ji}$ for $i=1, \ldots, p$.

\subsection{qL1M}

In addition to the Frobenius norm penalty, we could also consider other
types of penalty like the $l_1$ penalty, which gives the quantum $l_1$
minimization method (qL1M),
\begin{equation} \label{eq:ql1m}
    \min_{X\in \ob(n, p)} E_1(X) := \tr(\cgt{X}AX) + \mu_ 1 \sum_{i< j} \left| (\cgt{X}X)_{ij} \right|,
\end{equation}
with $\mu_1>0$ taking an appropriate value. Despite the nonsmooth nature
of the objective function introduced by the $l_1$ penalty term,
desirable properties can be established due to its well-known
``exactness" property \cite{nocedal1999numerical}. Specifically, given
the penalty parameter larger than a certain value, solving the
unconstrained problem with respect to $X$ yields the exact solution to
the original constrained problem. We now proceed to identify the local
minimizers for a detailed characterization of the landscape of qL1M.

\begin{theorem} \label{thm:local_min_ql1m}
    Local minimizers of qL1M \eqref{eq:ql1m} for Hermitian matrix $A$
    with $\mu_1 > 16p \|A\|_2$ and $p < n$ take the form $X=Q_p V^*$,
    where $V\in\bbC^{p\times p}$ is an arbitrary unitary matrix and
    $Q_p$ consists of eigenvectors of the $p$ smallest eigenvalues
    of $A$ (counting geometric multiplicity). Conversely, any matrix in
    the form of $Q_p V^*$ is a local minimizer of qL1M. Furthermore, any
    local minimum of qL1M is also a global minimum.
\end{theorem}

The proof of \cref{thm:local_min_ql1m} is given in \cref{sec:proof_ql1m}.

Note that large-scale eigenvalue problems typically aim to find only
a few most important eigenpairs. Hence, the constraint $p < n$ in
\cref{thm:local_min_ql1m} is acceptable in practice.

The theorem above demonstrates that the local minimizers of the qL1M
exhibit the same elegant structure as qOMM. Similarly, orthogonality of
the minimizers can naturally emerge throughout the optimization process
without any orthogonalization steps. This favorable structure indicates
the prospects for deriving computationally efficient algorithms for qL1M
that leverage established $l_1$ regularization techniques.

\begin{remark}
    Consider the weighted form of qL1M, which is defined as
    \begin{equation} \label{eq:weight_ql1m}
        \min_{X\in\ob(n,p)} \tr(\cgt{X}AX W) + \mu_1 \sum_{i< j} |(\cgt{X}X)_{ij}|,
    \end{equation}
    where weighted matrix $W=\diag(w_1, \cdots, w_p)$ such that $w_1 >
        w_2 > \cdots > w_p > 0$. One can easily verify that the proof of
    \cref{thm:local_min_ql1m} can be applied to the weighted qL1M
    \eqref{eq:weight_ql1m} and we conclude that the local minimum of
    \eqref{eq:weight_ql1m} with $\mu_1 > 0$ sufficiently large must be
    the local minimum of
    \begin{equation*}
        \min_{ X\in\bbC^{n\times p}} \tr(\cgt{X}AXW), \quad \text{s.t.} \quad \cgt{X}X=I.
    \end{equation*}
    Therefore, we have local minimizer $X=Q_p$ for the weighted qL1M,
    where $Q_p$ are ordered eigenvectors corresponding to the $p$
    smallest eigenvalues of $A$ (counting geometric multiplicity).
\end{remark}

Now let us briefly discuss the algorithm design for the qL1M. In the
quantum computing setting of qL1M, our objective function and
optimization problem \eqref{eq:ql1m} turn into
\begin{equation*}
    \min_{\theta_1, \cdots, \theta_p} \quad
    \sum_{i=1}^p \bra{\phi_i} \hU^\dagger_{\theta_i} \hH \hU_{\theta_i} \ket{\phi_i} + \mu_1
    \sum_{i,j=1, i<j}^p \left| \bra{\phi_i} \hU^\dagger_{\theta_i} \hU_{\theta_j} \ket{\phi_j} \right|,
\end{equation*}
where $\hH$ is the Hamiltonian operator of the system. Starting from
this objective function and following the same steps as qTPM, the desired
eigenpairs can be obtained. Additionally, more powerful and efficient
optimization algorithms could be developed by utilizing advanced
techniques for handling the $l_1$ penalty, though this lies beyond the
scope of this paper. The nonsmooth nature of the problem presents
opportunities for novel algorithm designs tailored to the qL1M
formulation, offering an interesting direction for future research.

\section{Numerical Results} \label{sec:numerical}

In this section, we conduct a comparative analysis of the three VQE
models: qOMM, qTPM, and qL1M, to assess their practical performance. The
comparison is twofold: we first evaluate the quantum resource cost
associated with each model, and then we present numerical simulations to
investigate their optimization behavior and convergence. These
results provide empirical validation for our theoretical findings and
clarify the trade-offs between the different approaches to enforce
orthogonality.

\subsection{Quantum resource estimate} \label{sec:quantum_resource_cost}

We first analyze the efficiency of the three VQE models by comparing
their quantum resource costs. We begin with the quantum circuits used to
evaluate state inner products. For two states
$\ket{\psi}=\hU_{\psi}\ket{0}$ and $\ket{\phi}=\hU_{\phi}\ket{0}$, the
complex-valued inner product $\langle \psi | \phi \rangle$ could be
evaluated via the extended Hadamard test \cite{bierman2022quantum},
which computes the real and imaginary parts separately with two
circuits. In contrast, the real-valued fidelity $|\langle \psi | \phi
    \rangle|^2$ could be measured with only one circuit, for instance, the
inversion test (measuring overlap $\bra{0} U_\psi^\dagger U_\phi
    \ket{0}$) or the SWAP test. While other variants exist, for the small
qubit systems considered here, these circuits are broadly comparable in
depth and qubit count. We therefore focus our analysis on the number of
inner product test circuits in each model, which is a primary indicator
of both circuit execution time and overall quantum resource expenditure.

The evaluation of the objective function of our VQE models involves
two parts: the Hamiltonian part and the regularization part. Assume that
the Hamiltonian $\hH$ is decomposed into a linear combination of $N_U$
Hermitian and unitary operators (e.g., Pauli strings) as $\hH =
    \sum_{k=1}^{N_U} c_k \hU_k$. Thus, the Hamiltonian term can be evaluated
through a general inner product circuit by $\langle \psi | \hU_k | \phi
    \rangle = \langle 0 | \hU_{\psi} \hU_k \hU_{\phi} | 0 \rangle$ for $k=1,
    \ldots, N_U$.

Specifically, the qOMM objective function
\begin{equation*}
    f(\ket{\psi_1}, \cdots, \ket{\psi_p})
    = 2 \sum_{i=1}^p \bra{\psi_i} \hH \ket{\psi_i} -
    \sum_{i,j=1}^p \langle \psi_i | \psi_j \rangle \bra{\psi_j} \hH \ket{\psi_i},
    \quad \hH = \sum_{k=1}^{N_U} c_k \hU_k,
\end{equation*}
involves complex-valued terms $\langle \psi_i | \psi_j \rangle$ and
$\langle \psi_i | \hU_k | \psi_j \rangle$ for $i < j$, $i,j=1,\ldots,
    p$, and real-valued terms $\langle \psi_i | \hU_k | \psi_i \rangle$ for
$i = 1, \ldots, p$. Consequently, we employ the extended Hadamard test,
requiring $p^2 N_U$ inner product test circuits for the Hamiltonian part
and $p(p-1)$ for the regularization part. In contrast, the objective
function of qTPM and qL1M involves only real-valued terms $|\langle
    \psi_i | \psi_j \rangle|^2$ for $i < j$, $i,j=1,\ldots, p$, and $\langle
    \psi_i | \hU_k | \psi_i \rangle$ for $i=1, \ldots, p$. Each Hamiltonian
term is evaluated via a single Hadamard test and each regularization
term via an inversion test, reducing the circuit counts to $p N_U$ and
$\frac{1}{2}p(p-1)$, respectively.

\Cref{tab:cost_analysis} summarizes the constituent terms of each
objective function and the corresponding number of inner product test
circuits required for their evaluation. The comparison reveals
a clear trade-off: qOMM, which employs an implicit regularization,
demands the highest quantum resource cost in terms of inner product test
circuits. In contrast, qTPM and qL1M, utilizing explicit regularization,
achieve a significant reduction in inner product test circuit overhead.
This quantum resource efficiency, however, comes at the cost of
increased complexity in classical optimization and hyperparameter selection.
Both qTPM and qL1M introduce hyperparameters that require careful
tuning, and their optimization landscapes present distinct challenges:
qL1M due to its non-smooth objective function and qTPM due to its more
complex local minima structure. Consequently, the selection of an
appropriate model necessitates a careful balance between quantum
resource constraints and classical optimization difficulty, tailored to
the specific application and available hardware.

\begin{table}[htbp]
    \centering
    \renewcommand{\arraystretch}{1.5}
    \begin{tabular}{c cc cc}
        \toprule
        \multirow{2}{*}{\textbf{VQE model}} & \multicolumn{2}{c}{\textbf{Hamiltonian cost}}         & \multicolumn{2}{c}{\textbf{Regularization cost}}                         \\
        \cmidrule(lr){2-3} \cmidrule(lr){4-5}
                                            & Terms                                                 & \#(Test circ)                                    & Terms & \#(Test circ) \\
        \hline
        qOMM
                                            & $\langle \psi_i | \hU_k | \psi_j \rangle$, $i \leq j$
                                            & $p^2 N_U$
                                            & $\langle \psi_i | \psi_j \rangle$, $i < j$
                                            & $p(p-1)$                                                                                                                         \\
        \hline
        qTPM
                                            & $\langle \psi_i | \hU_k | \psi_i \rangle$
                                            & $p N_U$
                                            & $|\langle \psi_i | \psi_j \rangle |^2$, $i < j$
                                            & $\frac{1}{2} p(p-1)$                                                                                                             \\
        \hline
        qL1M
                                            & $\langle \psi_i | \hU_k | \psi_i \rangle$
                                            & $p N_U$
                                            & $|\langle \psi_i | \psi_j \rangle|$, $i < j$
                                            & $\frac{1}{2} p(p-1)$                                                                                                             \\
        \bottomrule
    \end{tabular}
    \caption{Terms involved in the objective function and the number of
        inner product test circuits required for each VQE model. Here,
        $N_U$ denotes the number of terms in the Hamiltonian decomposition,
        $k=1, \ldots, N_U$ and $i, j=1, \ldots, p$.}
    \label{tab:cost_analysis}
\end{table}

\subsection{Optimization behavior}

In this section, we present numerical results from noiseless,
state-vector simulations for the proposed VQE models. All simulations
were implemented using the Qiskit package \cite{javadi2024quantum}. By
employing ideal quantum circuits and calculating exact probability and
expectation values, we focus on the models' optimization landscapes and
exclude the effects of quantum hardware noise and measurement
sampling.

We apply our VQE models on the task of finding the $p=3$ lowest
eigenvalues of the electronic structure Hamiltonian within the Full
Configuration Interaction (FCI) framework. For the molecule
$\moleculeHtwo$ with a distance $0.735$ \AA ~between two hydrogen atoms,
we construct the FCI Hamiltonian using PySCF \cite{sun2018pyscf} and map
it to a qubit Hamiltonian via the Jordan-Wigner transformation. This
results in a qubit Hamiltonian expressed as a linear combination of
Pauli operators, which defines the Hermitian matrix $A$ for our
eigenvalue computations. All reference eigenvalues are obtained using
PySCF's FCI solver.

Our numerical experiments employ the following setup for all three VQE
models. We model the $\moleculeHtwo$ molecule in a minimal STO-3G basis,
yielding a 4-qubit FCI Hamiltonian. The unitary coupled-cluster singles
and doubles (UCCSD) ansatz is adopted for the ansatz circuit, a standard
choice for quantum chemistry applications \cite{tilly2022variational}.
The quantum states before applying the ansatz are chosen as the
Hartree-Fock state and relevant single-excitation states. The ansatz
parameters are initialized to zero, which sets the UCCSD circuit to the
identity operator, which naturally prepares the initial states as the
Hartree-Fock state and relevant single-excitation states.

For the classical optimization, we use the derivative-free optimizers
COBYLA and BOBYQA for comparison. The penalty parameters are set to
$\mu=1.0$ for qTPM and $\mu_1=1.0$ for qL1M, values sufficiently large
to ensure the effectiveness of the regularization terms. The
optimization starts with the trust region radius set to
\texttt{rhobeg=1e-1} and is terminated when the trust region radius
falls below a minimum value of \texttt{rhoend=1e-7} or after a maximum
of 600 iterations.

To quantify convergence accuracy, we compute two relative error metrics
against reference solutions. The reference objective function values
$f_{\text{ref}}$ are the theoretical values at the local minima
(summarized in \Cref{tab:local_minima_summary}), and the reference
eigenvalues $\lambda_{\text{ref}}$ are obtained from PySCF. We measure
the relative errors in the objective function and the $l_2$-norm of the
eigenvalues, respectively, as:
\begin{equation*}
    \frac{|f_{\text{iter}} - f_{\text{ref}}|}{|f_{\text{ref}}|}
    \quad \text{and} \quad
    \frac{\|\lambda_{\text{iter}} - \lambda_{\text{ref}}\|_2}{\|\lambda_{\text{ref}}\|_2}.
\end{equation*}

\Cref{fig:convergence} displays the convergence behavior of all three
models using COBYLA (left panel) and BOBYQA (right panel)
optimizers, showing the relative error of the objective function versus
iteration count. The corresponding accuracy of the final converged
eigenvalues is quantified in \Cref{tab:final_eigen_err} through their
relative error to reference values.

\begin{figure}[htbp]
    \centering
    \includegraphics[width=0.9\textwidth]{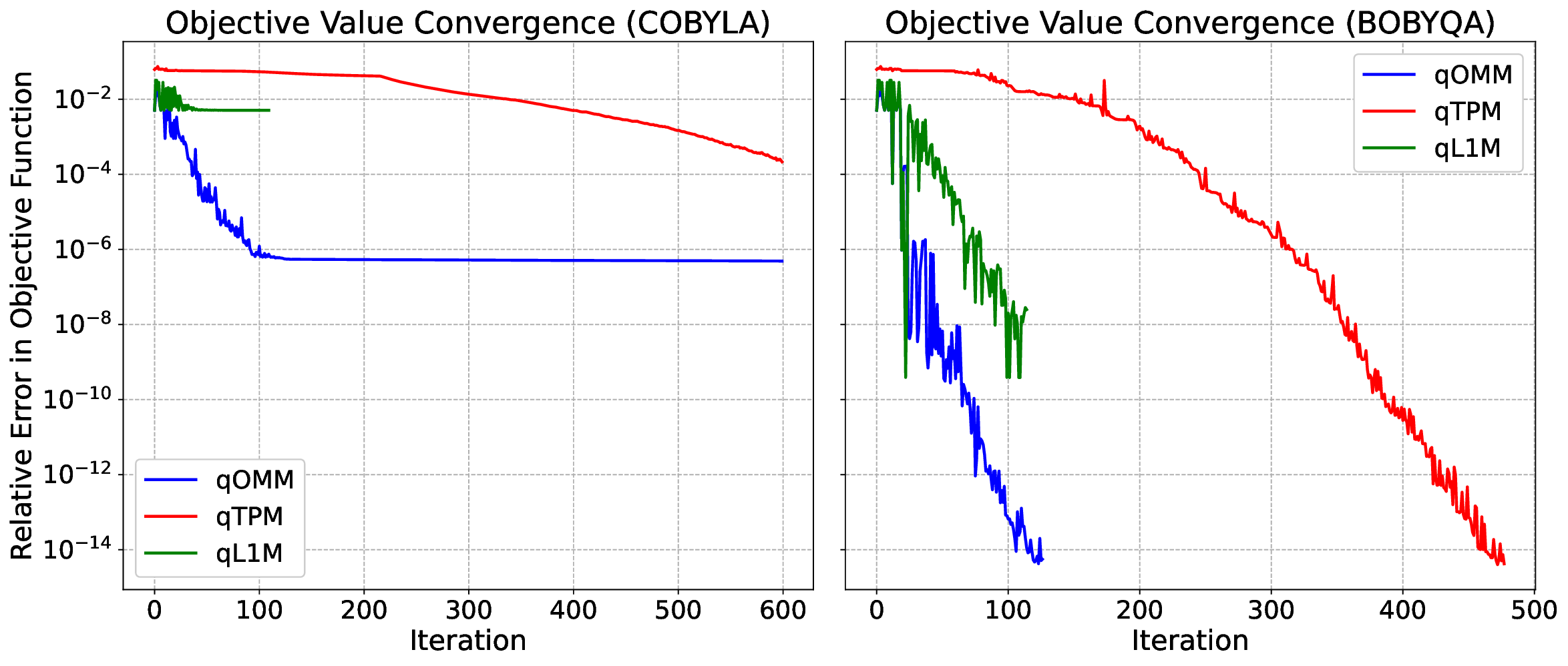}
    \caption{Relative error of the objective function versus
        optimization iteration of the three VQE models for the lowest three
        eigenvalues of the $\moleculeHtwo$ molecule with two derivative-free
        optimizers: COBYLA and BOBYQA. The stop criteria is set to a minimum
        trust region radius of \texttt{rhoend=1e-7} or a maximum of 600
        iterations.}
    \label{fig:convergence}
\end{figure}

\begin{table}[htbp]
    \centering
    \renewcommand{\arraystretch}{1.3}
    \begin{tabular}{c c c c}
        \toprule
        \textbf{VQE Model}    & \textbf{Eigenstate} & \textbf{COBYLA}        & \textbf{BOBYQA}        \\
        \midrule
        \multirow{3}{*}{qOMM} & Ground              & $1.83 \times 10^{-11}$ & $5.74 \times 10^{-15}$ \\
                              & 1st Excited         & $2.25 \times 10^{-10}$ & $8.03 \times 10^{-15}$ \\
                              & 2nd Excited         & $1.70 \times 10^{-8}$  & $1.13 \times 10^{-15}$ \\
        \midrule
        \multirow{3}{*}{qTPM} & Ground              & $2.47 \times 10^{-6}$  & $9.56 \times 10^{-16}$ \\
                              & 1st Excited         & $2.56 \times 10^{-6}$  & $9.63 \times 10^{-15}$ \\
                              & 2nd Excited         & $1.09 \times 10^{-4}$  & $1.38 \times 10^{-14}$ \\
        \midrule
        \multirow{3}{*}{qL1M} & Ground              & $1.09 \times 10^{-2}$  & $8.26 \times 10^{-10}$ \\
                              & 1st Excited         & $7.31 \times 10^{-15}$ & $4.46 \times 10^{-15}$ \\
                              & 2nd Excited         & $1.64 \times 10^{-15}$ & $8.80 \times 10^{-16}$ \\
        \bottomrule
    \end{tabular}
    \caption{The relative errors of the final computed eigenvalues
        for each model and optimizer, compared to the reference values
        from PySCF.}
    \label{tab:final_eigen_err}
\end{table}

\begin{figure}[htbp]
    \centering
    \includegraphics[width=0.9\textwidth]{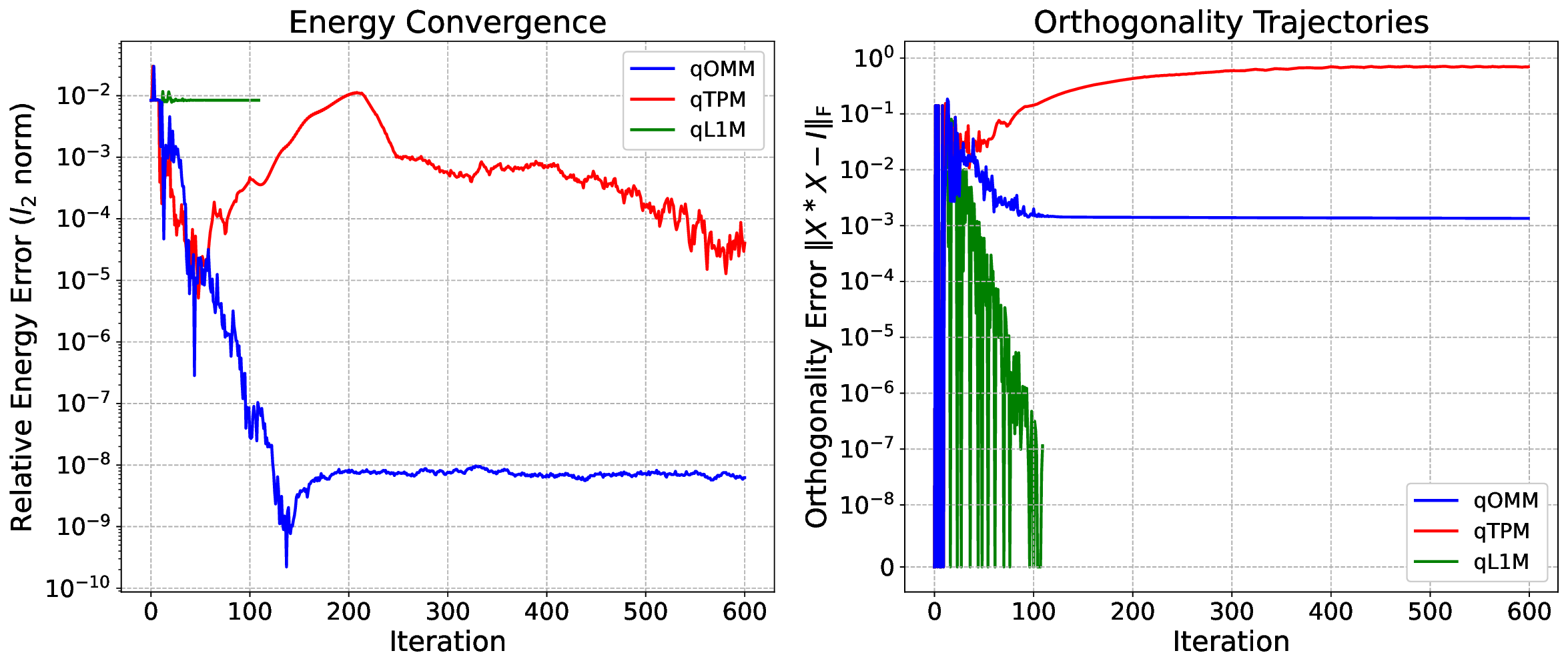}
    \caption{Convergence of relative errors of eigenvalues and
        trajectories of orthogonality constraint under COBYLA optimizer for
        three VQE models. (Left) The convergence behavior of relative error
        of eigenvalues in $l_2$ norm for all three models. (Right) The value
        of $\fnorm{X^*X - I}$ versus optimization iteration for three
        models.}
    \label{fig:energy_ortho_curves}
\end{figure}

For the accuracy of the final converged eigenvalues, all three models
achieve chemical accuracy (\texttt{1e-3} Hartree) compared to the
reference values, though qOMM and qTPM exhibit superior precision
compared to qL1M. As shown in \Cref{fig:convergence}, qL1M
is more sensitive to the choice of optimizer due to the non-smoothness
of its objective function. BOBYQA yields significantly better accuracy
in both the objective function value and computed eigenvalues than
COBYLA. This improvement may be attributed to BOBYQA's ability to
construct a quadratic model of the objective function, allowing for more
informed search directions even in the presence of non-smoothness.
Nevertheless, the inherent non-smoothness of qL1M remains a fundamental
challenge, making optimizer selection and parameter tuning more
demanding than for qOMM and qTPM when seeking comparable accuracy.

Furthermore, we note that the convergence speed varies among the three
models. Specifically, qOMM demonstrates the fastest convergence in terms
of number of iterations, while qTPM exhibits the slowest convergence
rate. This discrepancy stems from the different landscape structures of
the models. As previously discussed, qOMM's local minima directly align
with the orthogonal target eigen-subspace, together with a smooth
objective function, facilitating rapid convergence. In contrast, qTPM's
local minima are more complex, and it requires additional
refinement even after reaching the target subspace, which slows down the
overall convergence. qL1M, while also targeting the orthogonal subspace,
employs an $l_1$ penalty approach that may introduce additional
complexity in the optimization landscape, resulting in a moderate
convergence speed relative to the other two models.

To elucidate the connection between local minima structure and
optimization behavior, we examine the convergence of orthogonality error
and eigenvalue accuracy for all three VQE models using the COBYLA
optimizer (\Cref{fig:energy_ortho_curves}). The left panel displays the
corresponding eigenvalue errors obtained via Rayleigh-Ritz at each
iteration. Notably, qTPM exhibits a two-phase convergence: rapid
approach to the target eigen-subspace followed by prolonged refinement
toward its specific non-orthogonal minimum, requiring substantially more
iterations than needed for eigenvalue estimation alone. The right panel
tracks the evolution of $\fnorm{X^*X - I}$, revealing distinct
orthogonality enforcement patterns: qL1M rapidly converges to zero, qOMM
approaches orthogonality gradually, while qTPM converges to the
theoretically predicted non-zero value of $\fnorm{\Lambda_p -
    \bar{\Lambda}_p} / \mu$. The patterns of orthogonality trajectories
align with the convergence behaviors of the eigenvalues.

Theoretically, while all three models successfully identify the desired
low-lying eigen-subspace, their distinct local minima structures
(summarized in \Cref{tab:local_minima_summary}) lead to different
computational behaviors. The unitary solutions of qOMM and qL1M ($X^* X
    = I$) reduce the Rayleigh-Ritz procedure to a standard eigenvalue
problem, whereas qTPM's non-orthogonal solution ($X^*X = VSV^*$ with
diagonal $S$ and unitary $V$) necessitates a generalized eigenvalue
problem. Furthermore, qTPM's complex minima structure impacts its
convergence: even after rapidly identifying the target subspace from a
good initial point, the optimization continues to refine toward its
specific non-orthogonal minimum prescribed by the theory. This
additional adjustment, combined with the fact that optimization is
terminated based solely on the objective function's convergence or a
maximum iteration count, can result in a slower convergence rate for
qTPM compared to qOMM and qL1M.

\begin{table}[htbp]
    \centering
    \renewcommand{\arraystretch}{1.3}
    \begin{tabular}{c c c c}
        \toprule
        VQE model & Local minima form                                          & Objective value at local minima                                        \\
        \hline
        qOMM      & $X=Q_p V^*$                                                & $ \tr(\Lambda_p)$                                                      \\
        \hline
        qTPM      & $X=Q_p (I - (\Lambda_p - \bar{\Lambda}_p)/ \mu)^{1/2} V^*$ & $\tr(\Lambda_p) / 2  +  \tr(\bar{\Lambda}_p^2 - \Lambda_p^2) / (4\mu)$ \\
        \hline
        qL1M      & $X=Q_p V^*$                                                & $\tr(\Lambda_p)$                                                       \\
        \bottomrule
    \end{tabular}
    \caption{Summary of local minima for three VQE models. Here $Q_p$
        and $\Lambda_p$ are eigenpairs of matrix $A$ corresponding to the
        $p$ smallest eigenvalues, $\bar{\Lambda}_p =
            \frac{1}{p}\tr(\Lambda_p) I$, and $V\in\bbC^{p\times p}$ is
        unitary.}
    \label{tab:local_minima_summary}
\end{table}

Finally, we make a comparison for the entire quantum execution time and
hardware cost. The quantum resource analysis in
\cref{sec:quantum_resource_cost} indicates that qOMM requires
approximately $p$ times more quantum execution time per iteration than
qTPM and qL1M. However, \Cref{fig:convergence} reveals that for $p=3$,
qOMM requires only about (or less than) $1/p$ the number of iterations to
achieve a precision of \texttt{1e-14}. Consequently, the total quantum
execution time of qOMM is roughly comparable to (or even less than) that
of qTPM, though at the cost of greater quantum hardware resource
consumption. Meanwhile, qL1M achieves reduced quantum resource
requirements and execution time but demands more careful optimizer
selection and tuning to attain comparable accuracy, potentially
increasing classical computational overhead. These results underscore
that model selection involves a trade-off between quantum resource
efficiency and classical optimization complexity, which should be guided
by specific application requirements and hardware constraints.

\section{Conclusion} \label{sec:conclusion}

In this paper, we present the landscape analysis for three variational
quantum eigensolver (VQE) models designed for excited states
calculation: qOMM, qTPM, and qL1M. These models offer a solution to the
challenge of enforcing orthogonality between eigenstates, by embedding
the orthogonal constraint into the objective functions. These models
possess the desired property that any local minimum is also a global
minimum, which could be highly advantageous for optimization algorithm
design. Our analysis not only provides theoretical guarantees for these
favorable properties, but also offers valuable tools for landscape
analysis in various other VQE models.

Among the three models proposed, qOMM uses an implicit orthogonalization
approach by embedding the orthogonality constraint into the Hamiltonian
term, while qTPM and qL1M employ explicit regularization terms to
penalize non-orthogonality with different norms. The landscape analysis
reveals distinct characteristics for each model. For qOMM and qL1M, all
local minima correspond to orthogonal solutions, directly aligning with
the target eigen-subspace. In contrast, qTPM's local minima are
non-orthogonal but can be transformed into the desired eigen-subspace
through a generalized eigenvalue problem, which increases
the number of iterations. qL1M enforces orthogonality at its local
minima, similar to qOMM, but utilizes an $l_1$-norm based penalty, leading
to a non-smooth optimization landscape. This non-smoothness introduces
additional complexity in the optimization process, requiring specialized
algorithms to effectively navigate the landscape. The choice between
these models involves a trade-off between quantum resource demands and
classical optimization complexity, with qOMM being more measurement
intensive, while qTPM and qL1M offer reduced quantum costs at the
expense of hyperparameter tuning and more challenging optimization.

In addition, qOMM has been previously developed into a hybrid quantum-classical
algorithm with extensive numerical experiments demonstrating
favorable results. For qTPM and qL1M, the current work provides an
algorithmic framework with coarse-grained optimization strategies, while
leaving comprehensive numerical case studies for future work. We remark that with
techniques tailored to each model's particular structure, more efficient
and capable algorithms may be derived to enhance applicability in the
noisy intermediate-scale quantum era.

\appendix

\section{Proofs of qOMM} \label{sec:proof_qomm}

Let us first introduce the concept of strict majorization as follows.
Given a vector $x \in \bbR^n$, the notation $x^{\uparrow}$ denotes the
reordered vector of $x$ with entries in non-decreasing order, i.e.,
$x^{\uparrow}_1 \leq x^{\uparrow}_2 \leq \cdots \leq x^{\uparrow}_n$.

\begin{definition}[Strict Majorization]
    Let $d$ and $\lambda$ be two vectors in $\bbR^n$. The vector $\lambda$
    is strictly majorized by $d$, if
    \begin{equation} \label{eq:strict-major}
        \begin{split}
            \lambda^{\uparrow}_1                                     & < d^{\uparrow}_1, \\
            \lambda^{\uparrow}_1 + \lambda^{\uparrow}_2              &
            < d^{\uparrow}_1 + d^{\uparrow}_2,                                           \\
                                                                     & \,\,\,\vdots      \\
            \lambda^{\uparrow}_1 + \cdots + \lambda^{\uparrow}_{n-1} &
            < d^{\uparrow}_1 + \cdots + d^{\uparrow}_{n-1},                              \\
            \lambda^{\uparrow}_1 + \cdots + \lambda^{\uparrow}_{n-1}
            + \lambda^{\uparrow}_n                                   &
            = d^{\uparrow}_1 + \cdots + d^{\uparrow}_{n-1} + d^{\uparrow}_n.
        \end{split}
    \end{equation}
\end{definition}

\begin{lemma} \label{lem:majorization_cond}
    Assume vector $\lambda \in \bbR^{n}$ satisfies $\lambda \neq \mathbf{1}$
    and
    \begin{equation*}
        \sum_{i=1}^n \lambda_i = \sum_{i=1}^n 1 = n,
    \end{equation*}
    where $\mathbf{1}$ is the all-one vector in length $n$. Then $\lambda$
    is strictly majorized by $\mathbf{1}$.
\end{lemma}

\begin{proof}
    Without loss of generality, we assume that $\lambda$ is in non-decreasing
    order, i.e., $\lambda_1 \leq \lambda_2 \leq \cdots \leq \lambda_n$. Then
    we can further split $\lambda$ into three parts by comparing its
    components with $1$, i.e.,
    \begin{equation*}
        \lambda_1 \leq \cdots \leq \lambda_i < 1 = \lambda_{i+1} = \cdots
        = \lambda_{j} < \lambda_{j+1} \leq \cdots \leq \lambda_n.
    \end{equation*}
    Denote $d_1=d_2=\cdots = d_n=1$. Since $\lambda\neq\mathbf{1}$, the
    summation condition $\sum_{i=1}^n \lambda_i=\sum_{i=1}^n d_i$ guarantees
    that the first and the third part are not empty. Then for the first part
    of $\lambda$ whose values are strictly less than $1$, we have
    \begin{equation} \label{eq:lambda-strict-majorization}
        \lambda_1 + \cdots + \lambda_k < d_1 + \cdots + d_k,
    \end{equation}
    for all $k = 1, \ldots, i$. For the second part, $\lambda_i$s are all one.
    Hence, equation \eqref{eq:lambda-strict-majorization} is satisfied for $k
        = i+1, \dots, j$. Regarding the third part, where $\lambda_i$s are
    strictly greater than $1$, if there exists an index $\ell \in \{j+1,
        \dots, n-1\} $ such that
    \begin{equation*}
        \lambda_1 + \cdots + \lambda_\ell \geq d_1 + \cdots + d_\ell,
    \end{equation*}
    then we must have $\lambda_1 + \cdots + \lambda_s > d_1 + \cdots + d_s$
    for all $s > \ell$, which contradicts the summation condition
    $\sum_{i=1}^n \lambda_i = \sum_{i=1}^n d_i$. Combining the discussion above, 
    we complete the proof of the lemma.
\end{proof}

\begin{proof}[Proof of \cref{thm:stationary_qomm}]
    Stationary points of \eqref{eq:qomm} are characterized by the first-order
    condition of the Lagrangian function. Denoting the Lagrange multiplier by
    a diagonal matrix $D = \diag(d_1, \dots, d_p)$, the Lagrangian function admits
    \begin{equation} \label{eq:Lagrange_qomm}
        L(X, D) = \tr \left( (2I - X^*X) X^*AX \right)
        + \tr(D^\top (X^*X - I)).
    \end{equation}
    The first-order condition of \eqref{eq:Lagrange_qomm} yields
    \begin{subnumcases}{ \label{eq:Lagrange_qomm_zero_grad}}
        2AX - AXX^*X - XX^*AX + XD = 0, \label{eq:Lagrange_qomm_zero_grad_X} \\
        X \in \ob(n,p). \label{eq:Lagrange_qomm_zero_grad_D}
    \end{subnumcases}
    For any permutation matrix $P$, we have that
    \eqref{eq:Lagrange_qomm_zero_grad} is equivalent to
    \begin{equation*}
        \begin{cases}
            2AXP - AXP (XP)^*XP - XP (XP)^*AXP + XP P^\top DP = 0, \\
            XP \in \ob(n,p),
        \end{cases}
    \end{equation*}
    where $P^\top D P$ is a diagonal permutation of $D$. Therefore, without
    loss of generality, we assume that $D$ has $k$ distinct diagonal values in
    non-decreasing ordering,
    \begin{equation*}
        D =
        \begin{pmatrix}
            d_1 I_{p_1} &        &             \\
                        & \ddots &             \\
                        &        & d_k I_{p_k}
        \end{pmatrix}
    \end{equation*}
    for $d_1 < d_2 < \cdots < d_k$ and $p_i$ denoting the degeneracy of $d_i$.
    The matrix $X$ is denoted by $k$ column blocks, $X=\left( X_1, \dots, X_k
        \right)$ following the same partition of $D$. The reduced singular value
    decomposition of each $X_i$ admits $X_i = U_i \Sigma_i V_i^*$, where $U_i
        \in \bbC^{n\times r_i}, \Sigma_i \in \bbR^{r_i \times r_i}, V_i \in
        \bbC^{r_i \times p_i}$ for $r_i \leq p_i$ being the rank of $X_i$. Putting
    these SVDs together, we obtain the reduced SVD of $X$,
    \begin{equation} \label{eq:qomm_x_svd}
        \begin{split}
            X &=
            \begin{pmatrix}
                U_1 \Sigma_1 V_1^* & \cdots & U_k\Sigma_k V_k^*
            \end{pmatrix} \\
            &=
            \begin{pmatrix}
                U_1 & \cdots & U_k
            \end{pmatrix}
            \begin{pmatrix}
                \Sigma_1 &  & \\ & \ddots & \\ & & \Sigma_k
            \end{pmatrix}
            \begin{pmatrix}
                V_1^* &  & \\ & \ddots & \\ & & V_k^*
            \end{pmatrix}
            = U \Sigma V^*,            
            \end{split}
    \end{equation}
    where the diagonal of $\Sigma$ is not necessarily in non-increasing
    order.

    Left multiplying $X^*$ on both sides of
    \eqref{eq:Lagrange_qomm_zero_grad_X}, we have
    \begin{equation*}
        2X^*AX - X^*AX X^*X - X^*X X^*AX = -X^*XD,
    \end{equation*}
    whose left-hand side is Hermitian. Hence, the right-hand side $X^*XD$ is
    also Hermitian, i.e., $X^*XD=DX^*X$. Substituting the block structure of
    $X$ into $X^*XD=DX^*X$, we obtain,
    \begin{equation*}
        (d_i - d_j) X_j^* X_i = 0, \quad i,j=1, \ldots, k.
    \end{equation*}
    Since $d_i \neq d_j$ for $i \neq j$, we have $X_j^* X_i = 0$ and hence
    $U_j^* U_i = 0$.

    We left multiply $U^*$, right multiply $V$, substitute the SVD of $X$ into
    \eqref{eq:Lagrange_qomm_zero_grad_X}, and reorganize the expression
    leading to
    \begin{equation} \label{eq:zero_grad_svd}
        \tA \Sigma(2I - \Sigma^2) = \Sigma^2 \tA \Sigma - \Sigma D,
    \end{equation}
    where we denote $\tA =  U^* A U$ and adopt the fact $V^* D V = D$.
    Recalling the block structure of $X$ and $U$, we equate the left- and right-hand sides
    of \eqref{eq:zero_grad_svd}, and obtain relations for diagonal and
    off-diagonal blocks,
    \begin{equation} \label{eq:diag_block_compare}
        \tA_{ii} \Sigma_i(2I - \Sigma_i^2) = \Sigma_i^2 \tA_{ii} \Sigma_i
        - d_i \Sigma_i, \quad i = 1, \ldots, k,
    \end{equation}
    for diagonal blocks, and
    \begin{equation} \label{eq:off_diag_block_compare}
        \tA_{ij} \Sigma_j(2I- \Sigma_j^2) = \Sigma_i^2 \tA_{ij} \Sigma_j,
        \quad i, j = 1, \ldots, k, \text{ and } i \neq j,
    \end{equation}
    for off-diagonal blocks, where $\tA_{ij} = U_i^* A U_j$.

    We first focus on the $i$-th diagonal block in
    \eqref{eq:diag_block_compare}. Denote the $j$-th diagonal entry of
    $\tA_{ii}$ and $\Sigma_i$ as $\ta_{ij}$ and $\sigma_{ij}$, respectively.
    Comparing the diagonal entries in \eqref{eq:diag_block_compare} leads to
    \begin{equation*}
        2 \ta_{ij} (1 - \sigma_{ij}^2) = -d_i, \quad 1 \leq j \leq r_i.
    \end{equation*}
    By the negativity of $A$, we have $\ta_{ij}<0$ for all $i,j$, and
    represent $\sigma_{ij}$ by $d_i$ and $\ta_{ij}$ as
    \begin{equation*}
        \sigma_{ij}^2 = 1 + \frac{d_i}{2\ta_{ij}}.
    \end{equation*}
    Summing $j$ from $1$ to $r_i$, we obtain,
    \begin{equation*}
        d_i = \frac{2(p_i - r_i)}{\sum_{j=1}^{r_i}\frac{1}{\ta_{ij}}},
    \end{equation*}
    where we adopt the fact that $\tr(X_i^*X_i) = \sum_{j=1}^{r_i}
        \sigma_{ij}^2 = p_i$ by the first-order condition
    \eqref{eq:Lagrange_qomm_zero_grad_D}. When $p_i = r_i$, i.e., $X_i$ is
    of full rank, we have $d_i = 0$. Since $d_i$ and $d_j$ are distinct, we
    know that there is at most one $X_i$ of full rank and hence at most one
    block such that $\sigma_{ij} = 1$. For all other strictly low-rank
    blocks, we have $\sigma_{ij} > 1$.

    Next, we focus on the $(i, j)$ off-diagonal block in
    \eqref{eq:off_diag_block_compare}. Comparing the $(k, \ell)$ elements of
    the $(i, j)$ block, we have
    \begin{equation} \label{eq:uAu_zero}
        u_{ik}^* A u_{j \ell} \cdot
        \sigma_{j \ell} (2 - \sigma_{j \ell}^2 - \sigma_{i k}^2) = 0
        \Longrightarrow
        u_{ik}^* A u_{j \ell} = 0
    \end{equation}
    for $k = 1, \ldots, r_i$ and $\ell = 1, \ldots, r_j$, where $u_{ik}, u_{j
        \ell}$ are the $k$-th, $\ell$-th column of $U_i$, $U_j$ respectively.
    Equation \eqref{eq:uAu_zero} implies the $A$-orthogonality of $U_i$
    and $X_i$, i.e.,
    \begin{equation*}
        U_{i}^* A U_{j} = 0 \text{ and } X_i^* A X_j = 0,
    \end{equation*}
    for all $i \neq j$.

    As we have shown that both $X_i^* X_j$ and $X_i^* A X_j$ are zero for $i
        \neq j$, all blocks in \eqref{eq:Lagrange_qomm_zero_grad_X} are
    decoupled, and satisfy
    \begin{equation} \label{eq:Lagrange_qomm_zero_grad_block_X}
        2 A X_i - A X_i X_i^* X_i - X_i X_i^* A X_i + d_i X_i = 0,
        \quad i = 1, \ldots, k.
    \end{equation}

    When $X_i$ is of full-rank, we have $d_i = 0$, $\sigma_{ij} = 1$ for $j
        = 1, \dots, p_i$, and $X_i$ is a partial unitary matrix. Equation
    \eqref{eq:Lagrange_qomm_zero_grad_block_X} simplifies to
    \begin{equation*}
        A X_i = X_i X_i^* A X_i,
    \end{equation*}
    which implies that the column space of $X_i$ is an invariant subspace of $A$. Therefore,
    $X_i$ is spanned by some eigenvectors of $A$ and such $X_i$ satisfies
    the oblique manifold constraint.

    Now we consider the case that $X_i$ is strictly rank deficient. Given the
    basis $U_i$, the diagonal values of $\tA_{ii}$ are fixed. Substituting the
    expression of $d_i$ back into that of $\sigma_{ij}^2$, we have
    \begin{equation*}
        \sigma_{ij} =
        \sqrt{1 + \frac{p_i - r_i}{\ta_{ij} \cdot
                \sum_{\ell=1}^{r_i}\frac{1}{\ta_{i \ell}}}} > 1.
    \end{equation*}
    To satisfy the oblique manifold constraint, i.e., $X_i \in \ob(n, p_i)$,
    we need to find a $V_i$ such that $\diag(X_i^* X_i) = \mathbf{1}$. To
    achieve this, introduce the standard singular value decomposition of
    $X_i$ from its reduced version above as
    \begin{equation*}
        X_i = U_i \Sigma_i V_i^* = \begin{pmatrix}
            U_i & U_{i\bot}
        \end{pmatrix} \begin{pmatrix}
            \Sigma_{i} & 0 \\ 0 & 0
        \end{pmatrix} \begin{pmatrix}
            V_i^* \\ V_{i\bot}^*
        \end{pmatrix} := \bU_i \bSigma_i \bV_i^*,
    \end{equation*}
    where $\bU_i \in \bbC^{n \times p_i}$, $\bV_i \in \bbC^{p_i\times p_i}$.
    Denote $\sigma_{ij}=0$ for $j=r_{i+1}, \ldots, p_i$, then
    $\sum_{j=1}^{p_i} \sigma_{ij}^2=p_i$. By \cref{lem:majorization_cond},
    we know that $\{ \sigma_{ij}^2 \}_{j=1}^{p_i}$ and $\mathbf{1}$ satisfy
    the majorization relations. Therefore, by Schur-Horn theorem, there
    exists $\bV_i$ unitary such that
    $\diag(\bV_i\bSigma_i^2\bV_i^*)=\mathbf{1}$, which gives the desired
    $V_i$ such that $X_i \in \ob(n, p_i)$.

    Next we analyze $U_i$ of rank-deficient $X_i=U_i\Sigma_i V_i^*$. Denote
    the $j$-th column of $U_i$ as $u_{ij}$. Comparing off-diagonal term of
    \eqref{eq:diag_block_compare} we have
    \begin{equation*}
        u_{ij_1}^* A u_{ij_2} (2 - \sigma_{i j_2}^2-\sigma_{i j_1}^2) = 0,
        \quad j_1, j_2=1, \ldots, r_i, ~j_1 \neq j_2.
    \end{equation*}
    Note that $p_i > r_i$ implies $d_i < 0$ and hence $\sigma_{ij} > 1$ for
    $1 \leq j \leq r_i$. Therefore, $u_{ij_1}^* A u_{ij_2}=0$ for $j_1,
        j_2=1, \ldots, r_i, j_1 \neq j_2$, which gives
    \begin{equation*} \label{eq:U_diag_A}
        U_{i}^* A U_{i} = \diag\left(\ta_{i1}, \cdots, \ta_{i r_i}\right).
    \end{equation*}
    Besides, substitute $X_i=U_i\Sigma_i V_i^*$ into
    \eqref{eq:Lagrange_qomm_zero_grad_block_X}, multiply $V_i\Sigma_i^{-1}$
    from the right at both sides, we obtain
    \begin{equation*}
        AU_i(2I - \Sigma_i^2) = U_i(\Sigma_i^2 U_i^*AU_i - d_i I).
    \end{equation*}
    Since both $U_i^* A U_i$ and $2I-\Sigma_i^2$ are diagonal, given
    $j=1,\ldots, r_i$, if $\sigma_{ij}^2\neq 2$, $u_{ij}$ must be an
    eigenvector of $A$; if $\sigma_{ij}^2=2$, $u_{ij}$ may not be the
    eigenvector of $A$.

    Finally, we verify that any $X$ satisfying the conditions in
    \cref{thm:stationary_qomm} is indeed a stationary point of qOMM. Given
    the combined reduced SVD form $X=U \Sigma V^*$ as in
    \eqref{eq:qomm_x_svd}, according to conditions in
    \cref{thm:stationary_qomm}, 
    \begin{equation*}
        U^*AU:=\tA=\diag(\ta_{11}, \cdots, \ta_{1r_1}, \cdots, \ta_{k1}, \cdots, \ta_{kr_k}).     
    \end{equation*}
    Let
    \begin{equation*}
        D = \diag(d_1I_{p_1}, \cdots, d_k I_{p_k}), \quad
        d_i = \frac{2(p_i-r_i)}{\sum_{j=1}^{r_i} \frac{1}{\ta_{ij}}}, ~i=1, \ldots, k.
    \end{equation*}
    Then we have $\sigma_{ij}^2 = 1+ \frac{d_i}{2\ta_{ij}}$ and $d_i =
        2\ta_{ij}(\sigma_{ij}^2 -1)$ for $j=1, \ldots, r_i$, which gives
    \begin{equation} \label{eq:block_d_qomm}
        d_i I_{r_i} = 2\tA_i (\Sigma_i^2 - I_{r_i}),
    \end{equation}
    where $\tA_i=\diag(\ta_{i1}, \cdots, \ta_{ir_i})$ for $i=1, \ldots,
        k$. Note that for $V_i \in \bbC^{p_i \times r_i}$,
    \begin{equation*}
        V_i^* d_i I_{p_i} = d_i V_i^* = d_i I_{r_i} V_i^*,
    \end{equation*}
    thus we have $V^*D = D'V^*$, where $D' = \diag(d_1I_{r_1}, \cdots, d_k
        I_{r_k})$. Substitute $X=U\Sigma V^*$ into the first-order condition of stationary points
    \eqref{eq:Lagrange_qomm_zero_grad} we have
    \begin{equation*}
        2AX - AXX^*X - XX^*AX + XD
        = \left( AU(2I-\Sigma^2) - U(\Sigma^2 \tA - D') \right) \Sigma V^*.
    \end{equation*}
    Note that $AU(2I-\Sigma^2) - U(\Sigma^2 \tA - D')$ can be computed block
    by block according to partition $U=\left(U_1, \cdots, U_k\right)$. Take
    one block $U_i=\left(u_{i1}, \cdots, u_{ir_i}\right)$ for example,
    without loss of generality, we assume that $u_{i1}, \cdots, u_{i,
        r_i-1}$ are eigenvectors of $A$ except $u_{ir_i}$. 
    By the condition in \cref{thm:stationary_qomm} it leads to
    $\sigma_{ir_i}^2=2$. Therefore,
    \begin{align*}
         & \quad AU_i(2I-\Sigma_i^2)                                                                                               \\
         & = \left(u_{i1}, \cdots, u_{i,r_i-1}, Au_{ir_i}\right) \begin{pmatrix}
                                                                     \ta_{i1} &        &               &   \\
                                                                              & \ddots &               &   \\
                                                                              &        & \ta_{i,r_i-1} &   \\
                                                                              &        &               & 1
                                                                 \end{pmatrix} \begin{pmatrix}
                                                                                   2-\sigma_{i1}^2 &        &                      &   \\
                                                                                                   & \ddots &                      &   \\
                                                                                                   &        & 2-\sigma_{i,r_i-1}^2 &   \\
                                                                                                   &        &                      & 0
                                                                               \end{pmatrix} \\
         & = \left(u_{i1}, \cdots, u_{i,r_i-1}, u_{ir_i}\right) \begin{pmatrix}
                                                                    \ta_{i1} &        &               &            \\
                                                                             & \ddots &               &            \\
                                                                             &        & \ta_{i,r_i-1} &            \\
                                                                             &        &               & \ta_{ir_i}
                                                                \end{pmatrix} \begin{pmatrix}
                                                                                  2-\sigma_{i1}^2 &        &                      &   \\
                                                                                                  & \ddots &                      &   \\
                                                                                                  &        & 2-\sigma_{i,r_i-1}^2 &   \\
                                                                                                  &        &                      & 0
                                                                              \end{pmatrix}  \\
         & = U_i \tA_i (2I - \Sigma_i^2).
    \end{align*}
    Thus, according to \eqref{eq:block_d_qomm} we have
    \begin{equation*}
        AU_i(2I-\Sigma_i^2) - U_i(\Sigma_i^2 \tA_i - d_iI_{r_i}) =
        U_i(\tA_i (2I - \Sigma_i^2) - \Sigma_i^2\tA_i + d_i I_{r_i}) = 0,
    \end{equation*}
    which implies that
    \begin{equation*}
        2AX - AXX^*X - XX^*AX + XD =
        \left( AU(2I-\Sigma^2) - U(\Sigma^2 \tA - D') \right) \Sigma V^*=0.
    \end{equation*}
    Therefore, such $X$ is a stationary point of qOMM.
\end{proof}

To further distinguish the local minimizers from stationary points, we
introduce the concept of strong Schur-Horn continuity, which plays a key
role during the construction of descent direction near the stationary
point that is not a local minimizer.

\begin{definition}[Strong Schur-Horn Continuity \cite{chen2024continuity}] \label{def:sSHcont}
    Suppose $A \in \bbC^{n \times n}$ is a Hermitian matrix with an
    eigendecomposition $A = Q \Lambda Q^*$, where $Q$ is the unitary
    eigenvector matrix and $\Lambda$ is the diagonal eigenvalue matrix.
    Matrix $A$ is strongly Schur-Horn continuous if, for any perturbed
    eigenvalues $\tLambda$ satisfying $\tr(\tLambda) = \tr(\Lambda)$ and
    $\fnorm{\tLambda - \Lambda} = O(\veps)$ for $\veps > 0$ sufficiently
    small, there exists a Hermitian matrix $\tB = G_2 Q G_1 \tLambda G_1^*
        Q^* G_2^*$ such that
    \begin{enumerate}
        \item $\diag(\tB)=\diag(A)$,
        \item $G_1$ and $G_2$ are unitary matrices, and
        \item $\fnorm{G_i - I} = O(\veps^{1/2})$ for $i = 1, 2$.
    \end{enumerate}
\end{definition}

Furthermore, we have the following proposition for strong Schur-Horn
continuity.

\begin{proposition}[\cite{chen2024continuity}] \label{prop:sSHcont_strict_major}
    If $A\in \bbC^{n\times n}$ is a Hermitian matrix whose eigenvalue is
    strictly majorized by the diagonal, then $A$ is strongly Schur-Horn
    continuous.
\end{proposition}

\Cref{lem:oblique_perturb} essentially applies the strong Schur-Horn
continuity to singular values of a matrix, and obtains a construction for
perturbed singular value decomposition of a point in oblique manifold.

\begin{lemma} \label{lem:oblique_perturb}
    Suppose $X \in \ob(n, p)$ is non-unitary with a singular value
    decomposition $X=U \Sigma V^*$, where $U \in \bbC^{n\times p}$,
    $\Sigma \in \bbR^{p \times p}$ whose diagonal entries may not be in a
    non-increasing order and $V \in \bbC^{p \times p}$. Given any
    perturbed singular values $\tSigma \in \bbR^{p \times p}$ such that
    \begin{equation*}
        \tr(\tSigma^2) = \tr(\Sigma^2) \quad \text{and} \quad
        \fnorm{\tSigma^2 - \Sigma^2} = O(\veps),
    \end{equation*}
    for $\veps > 0$ sufficiently small, there exist unitary $\tV \in
        \bbC^{p \times p}$ such that $\tX := U \tSigma \tV^* \in \ob(n,p)$
    and
    \begin{equation*}
        \fnorm{\tX - X}= O(\veps^{1/2}).
    \end{equation*}
\end{lemma}

\begin{proof}
    Denote $H := X^*X = V \Sigma^2 V^*$ and $s := \diag(\Sigma^2)$. Since
    $X$ is a point on oblique manifold, we have $\sum_{i=1}^p s_i = \tr(H) =
        p$ and $\diag(H) = \mathbf{1}$. According to
    \cref{lem:majorization_cond}, $s$ is strictly majorized by $\mathbf{1}$. 
    Therefore,
    \cref{prop:sSHcont_strict_major} implies that $H$ is strongly Schur-Horn
    continuous, and there exists
    \begin{equation*}
        \tH = G_2 V G_1 \tSigma^2 G_1^* V^* G_2^*
    \end{equation*}
    such that $\diag(\tH)=\diag(H)=\mathbf{1}$, $G_i$ is unitary, and
    $\fnorm{G_i - I} = O(\veps^{1/2})$ for $i = 1, 2$. Let $\tV = G_2 V
        G_1$. Then, the distance between $V$ and $\tV$ is bounded by
    \begin{equation*}
        \fnorm{\tV - V} = \fnorm{G_2 V G_1 - G_2 V + G_2 V - V} \leq
        \fnorm{G_1 - I} + \fnorm{G_2 - I} = O(\veps^{1/2}),
    \end{equation*}
    where we use the unitary invariance of Frobenius norm. Constructing the
    perturbed $\tX$ as $\tX = U \tSigma \tV^*$, we know that $\tX \in
        \ob(n,p)$ following $\diag(\tX^* \tX) = \diag(\tH) = \mathbf{1}$.
    Furthermore, the distance between $X$ and $\tX$ could be bounded as,
    \begin{equation*}
        \fnorm{\tX - X} = \fnorm{U \tSigma \tV^* - U \Sigma V^*}
        \leq \fnorm{\tSigma - \Sigma}
        + \fnorm{\Sigma} \fnorm{\tV^* - V^*} = O(\veps^{1/2}),
    \end{equation*}
    which completes the proof.
\end{proof}

\begin{proof}[Proof of \cref{thm:local_min_qomm}]
    Given the combined reduced SVD of any stationary point $X=U\Sigma V^*$ as
    in \eqref{eq:qomm_x_svd}, one can always expand the SVD to include zero
    singular values, i.e.,
    \begin{equation*}
        X_i = U_i \Sigma_i V_i^* =
        \begin{pmatrix}
            U_i & U_{i\bot}
        \end{pmatrix}
        \begin{pmatrix}
            \Sigma_{i} & 0 \\ 0 & 0
        \end{pmatrix}
        \begin{pmatrix}
            V_i^* \\ V_{i\bot}^*
        \end{pmatrix}
        =: \bU_i \bSigma_i \bV_i^*,
    \end{equation*}
    where $\bU_i\in \bbC^{n\times p_i}$ is a partial unitary matrix,
    $\bSigma_i \in \bbR^{p_i \times p_i}$ is a diagonal matrix, $\bV_i \in
        \bbC^{p_i\times p_i}$ is a unitary matrix, and $\bU_i^* \bU_j = 0$ for all
    $1 \leq i \neq j \leq k$. Combining these SVDs similarly to that in
    \eqref{eq:qomm_x_svd}, we obtain,
    \begin{equation*}
        \begin{split}
            X &=
            \begin{pmatrix}
                \bU_1 \bSigma_1 \bV_1^* & \cdots & \bU_k \bSigma_k \bV_k^*
            \end{pmatrix} \\
            &=
            \begin{pmatrix}
                \bU_1 & \cdots & \bU_k
            \end{pmatrix}
            \begin{pmatrix}
                \bSigma_1 &  & \\ & \ddots & \\ & & \bSigma_k
            \end{pmatrix}
            \begin{pmatrix}
                \bV_1^* &  & \\ & \ddots & \\ & & \bV_k^*
            \end{pmatrix}
            = \bU \bSigma \bV^*.            
        \end{split}
    \end{equation*}
    Denoting the $j$-th diagonal entries of $\bU_i^* A \bU_i$ as $\ta_{ij}$
    for $j = 1, \ldots, p_i$ and $i = 1, \ldots, k$, the objective value at $X
        = \bU \bSigma \bV^*$ admits,
    \begin{equation*}
        E_0(X)
        = \sum_{i = 1}^k \tr\left( (2I - \bSigma_i^2)
        \bSigma_i \bU_i^* A \bU_i \bSigma_i \right)
        = \sum_{i = 1}^k \sum_{j = 1}^{r_i}
        \ta_{ij} (2-\sigma_{ij}^2)\sigma_{ij}^2.
    \end{equation*}
    Notice that $\ta_{ij}$s are all negative since $A$ is negative definite.
    In the following, we give the decay direction for low-rank $X$ and
    full-rank $X$, respectively.

    First, we consider stationary points $X$ that are of strict low-rank. In
    arbitrary $\veps$-neighborhood of $X$ for $\veps > 0$ sufficiently small,
    we construct a perturbed point $\tX = \bU \tSigma \tV^*$ such that
    \begin{equation*}
        \tSigma^2 = (1 - \veps)\Sigma^2 + \veps I.
    \end{equation*}
    One has $\tr(\tSigma^2)=\tr(\Sigma^2)$ and $\fnorm{\tSigma^2 - \Sigma^2} =
        O(\veps)$. According to \cref{lem:oblique_perturb}, there exists $\tV \in
        \bbC^{p\times p}$ such that $\tX = \bU \tSigma \tV^* \in \ob(n, p)$ and
    $\fnorm{\tX - X} = O(\veps^{1/2})$. The difference between the objective
    values at $\tX$ and $X$ is bounded as,
    \begin{equation*}
        E_0(\tX) - E_0(X) = \sum_{i=1}^{k} \sum_{j=1}^{r_i}
        \ta_{ij}(2 - \veps) \veps(\sigma_{ij}^2 -1)^2
        + \sum_{i=1}^{k} \sum_{j=r_i+1}^{p_i} \ta_{ij}(2 - \veps) \veps < 0.
    \end{equation*}
    Thus, for sufficiently small $\veps > 0$, there exists $\tX \in \ob(n,p)$
    such that $\fnorm{\tX - X} = O(\veps^{1/2})$ and $E_0(\tX) < E_0(X)$.
    The rank-deficient stationary points are strict saddle points.

    Next, we consider the full rank stationary point $X$. According to
    \cref{thm:stationary_qomm}, full rank stationary point is of form $X = U
        V^*$, with $U \in \bbC^{n \times p}$ being orthonormal eigenvectors of $A$
    and $V \in \bbC^{p \times p}$ is unitary. Denote the unitary eigenvector
    matrix of $A$ as $Q = \begin{pmatrix} q_1 & \cdots & q_n \end{pmatrix}$,
    whose corresponding eigenvalues are $\lambda_1 \leq \cdots \leq
        \lambda_n$. If $U$ is not spanned by eigenvectors corresponding to the $p$
    smallest eigenvalues, then there exist a column in $U$, denote as $u_s$,
    with eigenvalue $\tlambda > \lambda_p$, and another $q_t \not\in
        \mathrm{span}(U)$ with $t \leq p$ and $\lambda_t \leq \lambda_p <
        \tlambda$. We construct a perturbed point as
    \begin{equation*}
        \tX = \tU V^* =
        \begin{pmatrix}
            u_1     & \cdots & u_{s-1} & \sqrt{1-\veps^2} u_s + \veps q_t &
            u_{s+1} & \cdots & u_p
        \end{pmatrix}
        V^*.
    \end{equation*}
    We could verify that $\tX\in \ob(n,p)$ and $\fnorm{\tX - X}=O(\veps)$. The
    difference between objective values at $\tX$ and $X$ is bounded as,
    \begin{equation*}
        E_0(\tX) - E_0(X) = \veps^2(\lambda_t - \tlambda) < 0.
    \end{equation*}
    Thus, if $U$ is not spanned by eigenvectors corresponding to the $p$
    smallest eigenvalues, $X$ is not a local minimizer.

    We have previously demonstrated that any local minimizer of qOMM takes the
    form $X=Q_p V^*$, where $Q_p$ are eigenvectors of $A$ corresponding to
    the $p$ smallest eigenvalues $\Lambda_p$ and $V\in\bbC^{p\times p}$
    unitary. Substituting $X=Q_p V^*$ into the energy functional $E_0(X)$ it
    yields the same objective value $E_0(Q_p V^*)=\tr(\Lambda_p)$. Hence,
    any local minimum of qOMM is also a global minimum.

    Finally, note that oblique manifold $\ob(n, p)$ is a bounded and closed
    set and the objective function $E_0(X)$ is smooth with respect to $X$,
    there exists a minimum of $E_0(X)$ over $\ob(n,p)$. Thus, we conclude
    that any matrix in the form $X=Q_p V^*$ is a global minimizer of qOMM and
    there is no other local minimizer.
\end{proof}

\section{Proofs of qTPM} \label{sec:proof_qtpm}

The proofs of theorems of qTPM follow a similar procedure as qOMM above,
hence we present the main steps and emphasize the differences owing to
the nature of qTPM. Some detailed derivation may not appear again for
simplicity.

\begin{proof}[Proof of \cref{thm:stationary_qtpm}]
    The Lagrangian multiplier function of qTPM is
    \begin{equation*}
        \begin{split}
            L(X, D) := & g_{\mu}(X) + \frac{1}{2} \sum_{i=1}^p d_i(x_i^*x_i -1 ) \\
            =          & \frac{1}{2}\tr(X^* A X) + \frac{\mu}{4}\tr(X^*XX^*X)
            + \frac{1}{2}\tr(D^\top (X^*X - I)),
        \end{split}
    \end{equation*}
    then the first-order condition of stationary points gives
    \begin{subnumcases}{ \label{eq:Lagrange_qtpm_zero_grad}} 
        AX + \mu XX^*X + XD = 0, \label{eq:Lagrange_qtpm_zero_grad_X} \\
        X \in \ob(n,p). \label{eq:Lagrange_qtpm_zero_grad_D}
    \end{subnumcases}
    Note that $g_{\mu}(X)$ is invariant under right multiplication of
    permutation matrices. Similar to the proof of
    \cref{thm:stationary_qomm}, without loss of generality, we assume
    \begin{equation*}
        D=\diag \left( d_1 I_{p_1}, \cdots, d_k I_{p_k} \right),
    \end{equation*}
    with $d_1 < d_2 < \cdots < d_k$ and $p_1 + \cdots + p_k = p$.
    Correspondingly, we divide the columns of $X$ into $k$ parts $X=\left(
        X_1, \cdots, X_k \right)$ in the same way as $D$. Applying singular
    value decomposition to each block of $X$ one has $X_i = U_i \Sigma_i
        V_i^*$ where $U_i \in \bbC^{n\times p_i}, \Sigma_i\in\bbR^{p_i\times
            p_i}$ and $V_i \in \bbC^{p_i \times p_i}$ for $i=1, \ldots, k$. Thus,
    $$
        \begin{aligned}
            X & = \begin{pmatrix}
                      U_1 & \cdots & U_k
                  \end{pmatrix} \begin{pmatrix}
                                    \Sigma_1 &  & \\ & \ddots & \\ & & \Sigma_k
                                \end{pmatrix} \begin{pmatrix}
                                                  V_1^* &  & \\ & \ddots & \\ & & V_k^*
                                              \end{pmatrix} := U \Sigma V^*.
        \end{aligned}
    $$
    Multiplying $X^*$ from the left at both sides of
    \eqref{eq:Lagrange_qtpm_zero_grad_X} we have
    \begin{equation*}
        X^* A X + \mu X^*XX^*X + X^*XD =0,
    \end{equation*}
    which implies that $X^*XD = DX^*X$ and further, $(d_i - d_j)X_j^*X_i=0$
    for $i,j=1, \ldots, k$. Thus, we have $X_j^*X_i=0$ for $i\neq j$.
    Substitute the block partition of $X$ and $D$ into
    \eqref{eq:Lagrange_qtpm_zero_grad_X}, utilizing $X_j^*X_i=0$ for $i\neq
        j$ it yields
    \begin{equation} \label{eq:zero_grad_qtpm_block}
        AX_i + \mu X_i X_i^* X_i + d_i X_i = 0, \quad i=1, \ldots, k.
    \end{equation}
    Substitute $X_i=U_i\Sigma_i V_i^*$ into \eqref{eq:zero_grad_qtpm_block}
    for each block, multiplying $U_i^*$ from the left and multiplying
    $V_i$ from the right we arrive at
    \begin{equation} \label{eq:zero_grad_elem_qtpm}
        \tA_i \Sigma_i  + \mu  \Sigma_i^3 = - d_i \Sigma_i, \quad i=1, \ldots, k,
    \end{equation}
    where we denote $\tA_i = U_i^* A U_i$.

    To determine $U$, denote the $j$-th column of $U_i$ as $u_{ij}$,
    comparing off-diagonal terms of \eqref{eq:zero_grad_elem_qtpm} we have
    \begin{equation*}
        u_{ij_1}^* A u_{ij_2} \sigma_{i j_2} = 0, \quad
        j_1, j_2 =1, \ldots, r_i, ~j_1 \neq j_2.
    \end{equation*}
    Denote $\rank(X_i)=r_i$ for $i=1, \ldots, k$. Since $\sigma_{ij} >0$ for
    $1 \leq j \leq r_i$, $u_{ij_1}^* A u_{ij_2}=0$ for $ j_1, j_2 =1,
        \ldots, r_i, j_1 \neq j_2$. Denote the corresponding reduced singular
    value decomposition of $X_i$ as
    \begin{equation*}
        X_i=U_i \Sigma_i V_i^* =
        \begin{pmatrix} U_{i1} & U_{i2} \end{pmatrix}
        \begin{pmatrix} \Sigma_{i1} & 0 \\ 0 & 0 \end{pmatrix}
        \begin{pmatrix} V_{i1}^* \\ V_{i2}^* \end{pmatrix}
        =U_{i1} \Sigma_{i1} V_{i1}^*.
    \end{equation*}
    Then we know that $U_{i1}^* A U_{i1}$ is diagonal. Furthermore, since
    $X_i$ and $U_{i1}$ are invariant subspaces of matrix $A$ by
    \eqref{eq:zero_grad_qtpm_block}, $U_{i1}$ consists of mutually
    orthogonal eigenvectors of $A$. Besides, from $X_i^*X_j=0$ for $i\neq
        j$, we have $U_{i1}^* U_{j1}=0$. Thus, in order to give a compact form
    of $X$, we choose $U_{i2}$ corresponding to zero singular values for
    $i=1, \ldots, k$ to be eigenvectors of $A$ such that $U=\left(U_1,
        \cdots, U_k\right)$ is unitary, i.e., $U^*U=I$.

    Next we figure out the singular values of $X_i$ for $i=1, \cdots, k$.
    Denote the diagonal entries of $U_i^* A U_i$ as $\lambda_{ij}$ for $j=1,
        \ldots, p_i$. Comparing diagonal terms of \eqref{eq:zero_grad_elem_qtpm}
    we have
    \begin{equation*}
        \lambda_{ij}\sigma_{ij} + \mu \sigma_{ij}^3 =-d_i\sigma_{ij}, \quad j =1, \ldots, p_i.
    \end{equation*}
    Since $A$ is negative definite, $\lambda_{ij}<0$. Note that $\sigma_{ij}
        >0$ for $1 \leq j \leq r_i$, we have
    \begin{equation*}
        \sigma_{ij}^2 = -\frac{\lambda_{ij} + d_i}{\mu}, \quad j =1, \ldots, r_i.
    \end{equation*}
    Summing up $j$ from $1$ to $r_i$ we obtain,
    \begin{equation*}
        d_i = -\frac{1}{r_i} ( \sum_{j=1}^{r_i} \lambda_{ij} + \mu p_i ),
    \end{equation*}
    where we adopt the fact that $\tr(X_i^* X_i) = \sum_{j=1}^{r_i}
        \sigma_{ij}^2=p_i$ by the first-order condition
    \eqref{eq:Lagrange_qtpm_zero_grad_D}. Therefore, we get
    \begin{equation*}
        \sigma_{ij}^2 = \frac{p_i}{r_i}
        - \frac{1}{\mu}(\lambda_{ij} - \frac{1}{r_i}\sum_{l=1}^{r_i} \lambda_{il})
        := \frac{p_i}{r_i} - \frac{1}{\mu}\left( \lambda_{ij} - \blambda_i \right),
        \quad j=1, \ldots, r_i.
    \end{equation*}

    Now we check the consistency for stationary points satisfying the
    first-order condition \eqref{eq:Lagrange_qtpm_zero_grad}. Note that $A$
    is negative definite and $\mu > |\lambda_{\min}(A)|$, we have
    \begin{equation*}
        \mu > -\lambda_{\min}(A) \geq -\blambda_i \geq \frac{r_i}{p_i}(\lambda_{ij} - \blambda_i),
        \quad j=1, \ldots, r_i, i=1, \ldots, k,
    \end{equation*}
    which ensures that $\sigma_{ij}^2 > 0$ for $j=1, \ldots, r_i$ and $i=1,
        \ldots, k$. Given $i=1, \ldots, k$, if $\sigma_{ij}^2$ are not all ones
    for $j=1, \ldots, p_i$, according to \cref{lem:majorization_cond},
    $\{\sigma_{ij}\}_{j=1}^{p_i}$ and $\mathbf{1}$ satisfy the majorization
    relation. Thus, by Schur-Horn theorem, there exists
    $V_i\in\bbC^{p_i\times p_i}$ unitary such that $\diag(V_i \Sigma_i^2
        V_i^*)=\mathbf{1}$. Hence, we conclude that there always exists unitary
    $V_i$ such that $X_i=U_i \Sigma_i V_i$ satisfies the oblique manifold
    constraint.

    Finally, we verify that any $X$ satisfying the conditions in
    \cref{thm:stationary_qtpm} is indeed a stationary point of qTPM, i.e.,
    there exists a diagonal matrix $D$ such that
    \begin{equation*}
        AX + \mu XX^*X + XD = 0.
    \end{equation*}
    Given the singular value decomposition $X=U\Sigma V^*$ as in
    \cref{thm:stationary_qtpm}, let
    \begin{equation*}
        D = \diag(d_1I_{p_1}, \cdots, d_k I_{p_k}), \quad
        d_i = -\frac{1}{r_i}(\mu p_i+ \sum_{j=1}^{r_i} \lambda_{ij}), ~i=1, \ldots, k.
    \end{equation*}
    Then we have $V^*D=DV^*$. Since $AU=U\Lambda$, where
    $$ \Lambda=\diag\left(\lambda_{11}, \cdots, \lambda_{1p_1}, \cdots,
        \lambda_{k1}, \cdots, \lambda_{kp_k}\right), $$ 
    it yields
    \begin{equation*}
        AX + \mu XX^*X + XD = U\Sigma\left(\Lambda + \mu\Sigma^2 + D \right) V^*.
    \end{equation*}
    From the expression of $\Sigma$ in \cref{thm:stationary_qtpm}, one can
    easily see that $\Sigma\left(\Lambda + \mu\Sigma^2 + D \right)=0$ and
    hence $AX + \mu XX^*X + XD = 0$. Therefore, such $X$ is a stationary
    point of qTPM.
\end{proof}

\begin{proof}[Proof of \cref{thm:local_min_qtpm}]

    According to \cref{thm:stationary_qtpm}, stationary points of qTPM
    \eqref{eq:qtpm} take the form $X=U\Sigma V^*$ with $U^*U=I_p$ and
    $V^*V=I_p$. Besides, $U^* A U = \Lambda$ is diagonal with eigenvalues
    $\lambda_{11}, \cdots, \lambda_{1p_1}, \cdots, \lambda_{k1}, \cdots,
        \lambda_{kp_k}$ of $A$ being its diagonal entries. Then we can rewrite
    the objective value of qTPM at stationary point $X$ as
    \begin{align*}
        g_{\mu}(X) & = \frac{1}{2} \tr(V\Sigma U^*AU \Sigma V^*)
        + \frac{\mu}{4} \tr(V\Sigma U^* U\Sigma  V^* V\Sigma U^* U\Sigma V^*)         \\
                   & = \frac{1}{2}\tr(\Lambda \Sigma^2) + \frac{\mu}{4} \tr(\Sigma^4)
        = \frac{1}{2} \sum_{i=1}^k \sum_{j=1}^{p_i}
        (\lambda_{ij} + \frac{\mu}{2}\sigma_{ij}^2)\sigma_{ij}^2.
    \end{align*}
    Thus, the objective value at stationary points relies on the eigenvalues and
    singular values. In the following, we construct descent directions for
    stationary points that are strictly saddle points.

    First, we show that any rank deficient stationary point is a strict
    saddle point. Given the singular value decomposition of $X=U\Sigma V^*$
    as in \cref{thm:stationary_qtpm}, if $X$ is rank deficient, there exists
    at least a block of $\Sigma$, denoted as $\Sigma_i$, such that $r_i <
        p_i$. The diagonal entries of $\Sigma_i$ are
    \begin{equation*}
        \sigma_{ij} = \sqrt{\frac{p_i}{r_i} - \frac{1}{\mu} \left( \lambda_{ij} - \blambda_i \right)},
        \quad j=1, \ldots, r_i; \quad \sigma_{ij}=0, \quad j=r_{i+1}, \ldots, p_i,
    \end{equation*}
    where we denote $\blambda_i=\frac{1}{r_i} \sum_{l=1}^{r_i}
        \lambda_{il}$. Given $\veps>0$ sufficiently small, consider the
    perturbation $\tSigma$ of singular values $\Sigma$ by introducing
    \begin{equation*}
        \tsigma_{i1}^2 = \sigma_{i1}^2 - \veps, \quad
        \tsigma_{ip_i}^2 = \sigma_{ip_i}^2 + \veps = \veps,
    \end{equation*}
    where we adopt the fact that $\sigma_{ip_i}=0$. All the other entries in
    $\tSigma$ are the same as $\Sigma$. Then we have
    $\tr(\tSigma^2)=\tr(\Sigma^2)$ and $\fnorm{\tSigma^2 -
            \Sigma^2}=O(\veps)$. According to \cref{lem:oblique_perturb}, there
    exists $\tV\in\bbC^{n\times n}$ such that $\tX=U\tSigma \tV^* \in
        \ob(n,p)$ and $\fnorm{\tX - X}=O(\veps^{1/2})$. Thus, $\tX$ is a
    perturbation of $X$ over oblique manifold with desired singular values.
    At this time, we have
    \begin{align*}
        g_{\mu}(\tX) - g_{\mu}(X)                                                                                
        & = \frac{1}{2}(\lambda_{i1} + \frac{\mu}{2}\tilde{\sigma}_{i1}^2)\tilde{\sigma}_{i1}^2 -
        \frac{1}{2}(\lambda_{i1} + \frac{\mu}{2}\sigma_{i1}^2)\sigma_{i1}^2 +
        \frac{1}{2}(\lambda_{ip_i} + \frac{\mu}{2}\tilde{\sigma}_{ip_i}^2)\tilde{\sigma}_{ip_i}^2                    \\
        & = \frac{1}{2}\left(\lambda_{ip_i} - (\blambda_i+\frac{p_i}{r_i}\mu)\right) \veps + \frac{1}{2}\mu \veps^2.
    \end{align*}
    Note that $\lambda_{ip_i}<0$ since $A$ is negative definite. Besides,
    \begin{equation*}
        \blambda_i + \frac{p_i}{r_i} \mu > \blambda_i + \mu > \blambda_i + |\lambda_{\min}(A)| \geq 0.
    \end{equation*}
    Therefore, $g_{\mu}(\tX)-g_{\mu}(X) < 0$ when $\veps > 0$ is
    sufficiently small. Such rank deficient stationary point has a descent
    direction and is a strict saddle point.

    Next we turn to the full rank stationary points $X=U\Sigma V^*$.
    According to \cref{thm:stationary_qtpm}, the singular value matrix
    $\Sigma$ of full rank $X$ satisfies
    \begin{equation*}
        \Sigma^2 = I - \frac{1}{\mu} \begin{pmatrix}
            \Lambda_1 &        &           \\
                      & \ddots &           \\
                      &        & \Lambda_k
        \end{pmatrix} +
        \frac{1}{\mu} \begin{pmatrix}
            \blambda_1 I_{p_1} &        &                    \\
                               & \ddots &                    \\
                               &        & \blambda_k I_{p_k}
        \end{pmatrix},
    \end{equation*}
    where we denote the average value $\blambda_i= \frac{1}{p_i}
        \sum_{j=1}^{p_i} \lambda_{ij}$ for $i= 1, \ldots, k$. Note that those
    blocks with the same average value can be merged through a permutation.
    Without loss of generality, we assume $\blambda_i \neq \blambda_j$ for
    $i, j \in \{1, \ldots, k\}, i\neq j$. Now we show that any singular
    value matrix $\Sigma$ with $k> 1$ is a strict saddle point. Assume that
    $X$ is a full rank stationary point with $k > 1$, then one can always
    find two blocks with indices $i$ and $j$ such that $\blambda_i <
        \blambda_j$. Given $\veps > 0$ sufficiently small, consider a
    perturbation $\widetilde{\Sigma}$ of original $\Sigma$ by introducing
    \begin{equation*}
        \tilde{\sigma}_{i1}^2 = \sigma_{i1}^2 + \veps, \quad
        \tilde{\sigma}_{j1}^2 = \sigma_{j1}^2 - \veps,
    \end{equation*}
    while keeping all the other singular values in $\widetilde{\Sigma}$ the
    same as $\Sigma$. Then we have $\tr(\tSigma^2)=\tr(\Sigma^2)$ and
    $\fnorm{\tSigma^2 - \Sigma^2}=O(\veps)$. According to
    \cref{lem:oblique_perturb}, there exists $\tV\in\bbC^{n\times n}$ such
    that $\tX=U\tSigma \tV^* \in \ob(n,p)$ and $\fnorm{\tX -
            X}=O(\veps^{1/2})$. Thus, $\tX$ is a perturbation of $X$ over oblique
    manifold with desired singular values. At this time, we have
    \begin{align*}
          & g_{\mu}(\tX) - g_{\mu}(X)                                                             \\
        = & \frac{1}{2}(\lambda_{i1} + \frac{\mu}{2}\tilde{\sigma}_{i1}^2)\tilde{\sigma}_{i1}^2 -
        \frac{1}{2}(\lambda_{i1} + \frac{\mu}{2}\sigma_{i1}^2)\sigma_{i1}^2 +
        \frac{1}{2}(\lambda_{j1} + \frac{\mu}{2}\tilde{\sigma}_{j1}^2)\tilde{\sigma}_{j1}^2 -
        \frac{1}{2}(\lambda_{j1} + \frac{\mu}{2}\sigma_{j1}^2)\sigma_{j1}^2                       \\
        = & \frac{1}{2}(\blambda_i - \blambda_j) \veps + \frac{1}{2}\mu \veps^2 < 0,
    \end{align*}
    when $\veps > 0$ is sufficiently small. Thus, those full rank stationary
    points $X$ with $k>1$ are all strict saddle points.

    Finally, we consider the selection of eigenvectors. We have shown that
    any local minimizer has a singular value decomposition $X=U \Sigma V^*$
    where $U$ consists of eigenvectors of matrix $A$, $V\in \bbC^{n\times
            n}$ unitary such that $\diag(X^*X)=\mathbf{1}$. Denote the ordered
    eigenvectors of $A$ as $\begin{pmatrix} q_1 & \cdots & q_n\end{pmatrix}$
    with corresponding eigenvalues $\lambda_1 \leq \cdots \leq \lambda_n$.
    If $U$ is not spanned by eigenvectors corresponding to the $p$ most
    smallest eigenvalues, then there exists a column in $U$, denoted as
    $u_s$, with eigenvalue $\tlambda > \lambda_p$, and another $q_t\notin
        \text{span}(U)$ with $t \leq p$ and $\lambda_t \leq \lambda_p <
        \tlambda$. We construct a perturbed point as
    \begin{equation*}
        \tX = \tU \Sigma V^* = \begin{pmatrix}
            u_1 & \cdots & u_{s-1} & \sqrt{1-\veps^2} u_s + \veps q_t & u_{s+1} & \cdots & u_p
        \end{pmatrix} \Sigma V^*.
    \end{equation*}
    One could verify that $\tX\in \ob(n,p)$ and $\fnorm{\tX - X}=O(\veps)$.
    The difference between objective values at $\tX$ and $X$ is bounded as,
    \begin{equation*}
        g_{\mu}(\tX) - g_{\mu}(X) = \frac{1}{2} \veps^2 (\lambda_p - \tlambda) \sigma_s^2 < 0.
    \end{equation*}
    Thus, if $U$ is not spanned by eigenvectors corresponding to the $p$
    smallest eigenvalues, $X$ is not a local minimizer.

    Therefore, any local minimizer of qTPM takes the form $X=Q_p (I -
        \frac{1}{\mu}(\Lambda_p - \bar{\Lambda}_p))^{1/2} V^*$, where $Q_p$ and
    $\Lambda_p$ are eigenpairs of $A$ corresponding to the $p$ smallest
    eigenvalues, $\bar{\Lambda}_p=\frac{1}{p}\tr(\Lambda_p) I$. Substituting
    $X=Q_p (I - \frac{1}{\mu}(\Lambda_p - \bar{\Lambda}_p))^{1/2} V^*$ into
    the energy functional $g_{\mu}(X)$ yields the same objective value
    \begin{equation*}
        g_{\mu} \left( Q_p(I - \frac{1}{\mu}(\Lambda_p - \bar{\Lambda}_p))^{1/2} V^* \right)
        = \frac{\mu}{4} \tr\left((I + \frac{1}{\mu}\bar{\Lambda}_p)^2 - \frac{1}{\mu^2} \Lambda_p^2 \right).
    \end{equation*}
    Again, since oblique manifold $\ob(n, p)$ is a bounded and closed set
    and the objective function $g_{\mu}(X)$ is smooth with respect to $X$,
    there exists a global minimum of $g_{\mu}(X)$ over $\ob(n,p)$. Thus, we
    conclude that any matrix in the form $X=Q_p (I - \frac{1}{\mu}(\Lambda_p
        - \bar{\Lambda}_p))^{1/2} V^*$ is a global minimizer of qTPM and there is
    no other local minimizer.
\end{proof}

\section{Proofs of qL1M} \label{sec:proof_ql1m}

\begin{lemma} \label{lem:descent_d_ql1m}
    Given matrix $X=(x_1, \cdots, x_p)\in\bbC^{n\times p}$ and vector
    $x_0\in \bbC^{n\times 1}$ such that $x_0 \notin \text{span}(X)$,
    there exists $d\in \bbC^{n\times 1}$ such that
    \begin{equation*}
        d^* d = 1, \quad d^* X =0, \quad d^* x_0 < 0.
    \end{equation*}
    Additionally, if for $i,j =0, 1, \cdots, p$, $i\neq j$,
    \begin{equation} \label{eq:descent_d_condition}
        x_i^* x_i =1, \quad |x_i^* x_j| \leq \veps_0,
    \end{equation}
    where $\veps_0 = \frac{1}{4p}$, then there exists
    $d\in\bbC^{n\times 1}$ such that
    \begin{equation*}
        d^* d = 1, \quad d^* X = 0, \quad d^* x_0 < -\frac{1}{2}.
    \end{equation*}
\end{lemma}

\begin{proof}
    Denote the orthogonal projection of $x_0$ over $\text{span}(X)$ as
    $s=XX^{\dagger}x_0$. Since  $x_0\notin \text{span}(X)$, let
    \begin{equation*}
        d = \frac{s- x_0}{\twonorm{s - x_0}},
    \end{equation*}
    then we have
    \begin{equation*}
        d^* d = 1, \quad d^* X = 0, \quad d^* x_0 = d^*(s- \twonorm{s-x_0}d) = -\twonorm{s-x_0} < 0,
    \end{equation*}
    which completes the proof of the first statement.

    Next we give a lower bound of $\twonorm{s-x_0}$ based on condition
    \eqref{eq:descent_d_condition}. According to Gerschgorin disk
    theorem, eigenvalues of $X^*X$ lie in disks
    \begin{equation*}
        G_i = \{ z\in\bbC : |z-1| \leq \sum_{j\neq i}^{p} |(X^* X)_{ij}| \leq p\veps_0 < \frac{1}{2} \}, \quad i=1, \cdots, p.
    \end{equation*}
    Thus, $X^*X$ is non-singular and all the eigenvalues of $X^* X$ lie
    in $[\frac{1}{2}, \frac{3}{2}]$. Moreover, all the eigenvalues of
    $(X^*X)^{-1}$ lie in $[\frac{2}{3}, 2]$. Therefore,
    \begin{equation*}
        \|X\|_2 \leq \sqrt{\frac{3}{2}}, \quad \|(X^* X)^{-1}\|_2 \leq 2.
    \end{equation*}
    Note that $|(X^*x_0)_j| = |x_j^* x_0| \leq \veps_0$ for $j=1,
        \ldots, p$. At this time, we have
    \begin{equation*}
        \twonorm{s} = \twonorm{X(X^*X)^{-1}X^* x_0}
        \leq \twonorm{X} \twonorm{(X^*X)^{-1}} \twonorm{X^* x_0}
        \leq \sqrt{6p} \veps_0 \leq \frac{\sqrt{3}}{2}.
    \end{equation*}
    Hence,
    \begin{equation*}
        d^*x_0 = -\twonorm{s-x_0} = -\left( \twonorm{x_0}^2 - \twonorm{s}^2 \right)^{1/2} \leq -\frac{1}{2},
    \end{equation*}
    which completes the proof of the second statement.
\end{proof}

\begin{proof}[Proof of \cref{thm:local_min_ql1m}]

    Motivated by the exact property of $l_1$ penalty, we complete the proof
    by showing that any local minimizer of qL1M \eqref{eq:ql1m} actually
    satisfies the orthogonal constraint $\cgt{X}X=I$ and thus becomes the
    local minimizer of classical trace minimization method. Specifically, for
    any $X$ not satisfying the orthogonal constraint, one can always find a
    descent direction nearby over the oblique manifold, indicating that it
    cannot be a local minimizer.

    Denote $X =\left(x_1, \cdots, x_p\right)$. Assume that $X$ does not
    satisfy the orthogonal constraint, i.e., there exists columns $x_i,
        x_j$, $i\neq j$ such that $\cgt{x_i}x_j\neq 0$. Now we find a descent
    direction for such $X$ based on the value of $|\cgt{x}_ix_j|$. Given a
    constant $\veps_0 = \frac{1}{4p}$, we divide the discussion into two
    scenarios.

    \emph{Case 1: There exists $i, j\in\{1, \ldots, p\}$, $i\neq j$ such
        that $|\cgt{x_i}x_j| > \veps_0$.}

    Without loss of generality, we assume that $|\cgt{x_1} x_2| > \veps_0$.
    Consider a perturbation of $X$ over oblique manifold denoted as $\tX
        =(\tx_1, \cdots, \tx_p)$, such that
    \begin{equation*}
        \tx_1 = \sqrt{1 -\veps^2} x_1 + \veps d, \quad \tx_j = x_j, ~j = 2, 3 \ldots, p,
    \end{equation*}
    with $\veps > 0$ sufficiently small and $d\in \bbC^{n\times 1}$ to be
    determined later. Note that $p < n$, if $Ax_1 \in \text{span}(X)$, there
    exists $d$ such that
    \begin{equation*}
        \cgt{d}d=1, \quad \cgt{d}X=0, \quad \text{and} \quad \cgt{d}Ax_1 =0.
    \end{equation*}
    If $Ax_1 \notin \text{span}(X)$, according to \cref{lem:descent_d_ql1m},
    there exists $d$ such that
    \begin{equation*}
        \cgt{d}d=1, \quad \cgt{d}X=0, \quad \text{and} \quad \cgt{d}Ax_1 < 0.
    \end{equation*}
    In both cases we have
    \begin{equation*}
        \cgt{\tx_1} \tx_1 = 1, \quad
        \cgt{\tx_1}\tx_j = \sqrt{1-\veps^2}\cgt{x_1}x_j, ~j=2, \ldots, p,
    \end{equation*}
    and
    \begin{equation*}
        \cgt{\tx_1}A\tx_1 = (1-\veps^2) \cgt{x_1}Ax_1 + 2\veps\sqrt{1-\veps^2} \rep(\cgt{d}Ax_1) + \veps^2 \cgt{d} A d.
    \end{equation*}
    Note that $\lambda_{\min}(A) \leq \cgt{x_1} A x_1 \leq
        \lambda_{\max}(A)$ and $\lambda_{\min}(A) \leq \cgt{d} A d \leq
        \lambda_{\max}(A)$, together with $d^* A x_1 \leq 0$, we have
    \begin{align*}
        \qquad & E_1(\tX) - E_1(X) \\
        = & -\veps^2 \cgt{x_1}Ax_1 + 2\veps\sqrt{1-\veps^2}\cgt{d}Ax_1 + \veps^2 \cgt{d} A d + \mu_1 \sum_{j\neq 1} (\sqrt{1-\veps^2} - 1) |\cgt{x_1}x_j| \\
        \leq & \left(\lambda_{\max}(A) - \lambda_{\min}(A)\right) \veps^2 + \mu_1 (\sqrt{1 - \veps^2} - 1) \veps_0                                        \\
        \leq & \left( \lambda_{\max}(A) - \lambda_{\min}(A) - \mu_1 \frac{\veps_0}{2} \right) \veps^2,
    \end{align*}
    where we use $\sqrt{1-\veps^2} - 1 \leq -\frac{\veps^2}{2}$ for $\veps >
        0$ sufficiently small in the last inequality. Further, from
    \begin{equation*}
        \mu_1 > 16 p \|A\|_2 = \frac{4}{\veps_0} \|A\|_2 \geq \frac{2}{\veps_0}(\lambda_{\max}(A) - \lambda_{\min}(A)),
    \end{equation*}
    we have $E_1(\tX) - E_1(X) <0$, while
    \begin{equation*}
        \fnorm{\tX - X} = \left\|(\sqrt{1-\veps^2} - 1) x_1 + \veps d \right\|_2 \leq \veps.
    \end{equation*}
    Therefore, such $X$ cannot be a local minimizer.

    \emph{Case 2: $|\cgt{x_i}x_j| \leq \veps_0$, for $i, j=1, \ldots, p$,
        $i\neq j$.}

    Without loss of generality, we assume that $|\cgt{x_1}x_2| \neq 0$. Let
    $\mi=\sqrt{-1}$, denote
    \begin{equation*}
        \cgt{x_1}x_2 = a + \mi b, \quad \text{where} \quad
        a=\rep(\cgt{x_1}x_2), b=\imp(\cgt{x_1}x_2).
    \end{equation*}
    From $(\frac{a}{\sqrt{a^2+b^2}})^2 + (\frac{b}{\sqrt{a^2+b^2}})^2 = 1$,
    we know that one of $|\frac{a}{\sqrt{a^2+b^2}}|$ or
    $|\frac{b}{\sqrt{a^2+b^2}}|$ is greater than $\frac{1}{\sqrt{2}}$. This
    condition further separates the discussion into two subcases.

    We firstly address the case when $|\frac{a}{\sqrt{a^2+b^2}}| \geq
        \frac{1}{\sqrt{2}}$. Here we assume $a=\rep(\cgt{x_1}x_2)>0$, the
    analysis for $a <0$ can be carried in the same manner. According to
    \cref{lem:descent_d_ql1m}, there exists $d\in \bbC^{n\times 1}$ such
    that
    \begin{equation*}
        \cgt{d}d =1, \quad \cgt{d}x_1 = 0, \quad \cgt{d}x_2 < -\frac{1}{2}, \quad \text{and} \quad \cgt{d}x_j = 0, ~j =3, \ldots, p.
    \end{equation*}
    Consider a perturbation of $X$ over oblique manifold, which is denoted
    as $\tX =(\tilde{x}_1, \cdots, \tilde{x}_p)$, such that
    \begin{equation*}
        \tilde{x}_1=\sqrt{1-\veps^2}x_1 + \veps d, \quad \tilde{x}_j = x_j, ~j =2, \ldots, p.
    \end{equation*}
    Denote $c=\cgt{d}x_2 < 0$, one has
    \begin{align*}
         & \cgt{\tilde{x}_1} \tilde{x}_1 = 1,                                                                                 \\
         & \cgt{\tilde{x}_1}\tx_2 = \sqrt{1-\veps^2}\cgt{x_1}x_2 + \veps \cgt{d}x_2 = \sqrt{1-\veps^2} (a + \mi b) + \veps c, \\
         & \cgt{\tilde{x}_1}\tx_j = \sqrt{1-\veps^2}\cgt{x_1}x_j, \quad j=3, \ldots, p.
    \end{align*}
    Note that $a > 0$, $c < 0$, let $\veps > 0$ be small enough such that
    $\sqrt{1-\veps^2} a + \veps c > 0$. Thus we have
    \begin{equation*}
        |\cgt{\tilde{x}_1} \tilde{x}_2| = |\sqrt{1-\veps^2}\cgt{x_1}x_2 + \veps \cgt{d}x_2| < |\cgt{x_1}x_2|.
    \end{equation*}
    Hence,
    \begin{align*}
        |\sqrt{1-\veps^2}\cgt{x_1}x_2 + \veps \cgt{d}x_2| - |\cgt{x_1}x_2|
         & = \frac{|\sqrt{1-\veps^2}\cgt{x_1}x_2 + \veps \cgt{d}x_2|^2 - |\cgt{x_1}x_2|^2}{|\sqrt{1-\veps^2}\cgt{x_1}x_2 + \veps \cgt{d}x_2| + |\cgt{x_1}x_2|} \\
         & \leq \frac{|\sqrt{1-\veps^2}\cgt{x_1}x_2 + \veps \cgt{d}x_2|^2 - |\cgt{x_1}x_2|^2}{2|\cgt{x_1}x_2|}                                                 \\
         & = \frac{|\sqrt{1-\veps^2}a + \veps c + \mi \sqrt{1-\veps^2} b|^2 - |a+\mi b|^2}{2\sqrt{a^2+b^2}}                                                    \\
         & = \frac{1}{2\sqrt{a^2+b^2}} \left( \veps^2(c^2 - a^2 - b^2) + 2\veps\sqrt{1-\veps^2} ac \right)                                                     \\
         & \leq \frac{\veps^2c^2}{2\sqrt{a^2+b^2}} + \frac{\veps c}{2},
    \end{align*}
    where in the last equality we use $\frac{a}{\sqrt{a^2+b^2}} \geq
        \frac{1}{\sqrt{2}}$ and $\sqrt{1-\veps^2} \geq \frac{1}{\sqrt{2}}$ for
    $\veps > 0$ sufficiently small. Additionally, from
    \begin{equation*}
        \cgt{\tilde{x}_1} A \tilde{x}_1 = (1 - \veps^2)\cgt{x_1} A x_1 + 2\veps\sqrt{1-\veps^2} \rep( \cgt{d}Ax_1) + \veps^2 \cgt{d} A d,
    \end{equation*}
    we have
    \begin{align*}
        E_1(\tX) - E_1(X) & \leq -\veps^2 \cgt{x_1} A x_1 + 2\veps\sqrt{1-\veps^2} \rep( \cgt{d}Ax_1) + \veps^2 \cgt{d} A d                                       \\
                          & \quad + \mu_1 (\frac{\veps^2c^2}{2\sqrt{a^2+b^2}} + \frac{\veps c}{2}) +  \mu_1 \sum_{j = 3}^p (\sqrt{1-\veps^2} - 1) |\cgt{x_1}x_j|.
    \end{align*}
    Note that
    \begin{equation*}
        |\cgt{d}Ax_1| \leq \|d\|_2 \cdot \|A\|_2 \cdot \|x_1\|_2 \leq \|A\|_2,
    \end{equation*}
    together with $c \leq -\frac{1}{2}$ we have
    \begin{equation} \label{eq:energy_diff_ql1m}
        E_1(\tX) - E_1(X) \leq \veps(2\|A\|_2 - \frac{\mu_1}{4}) + \veps^2 (\cgt{d}Ad - \cgt{x_1}Ax_1 + \frac{\mu_1 c^2}{2\sqrt{a^2+b^2}}).
    \end{equation}
    Since $\mu_1 > 16 p \|A\|_2$, it implies that $2\|A\|_2 -
        \frac{\mu_1}{4} < 0$. Besides, the right-hand side of
    \eqref{eq:energy_diff_ql1m} will be dominated by the linear term when
    $\veps > 0$ is sufficiently small. Hence, we have $E_1(\tX) - E_1(X) <
        0$ with $\fnorm{\tX-X}\leq \veps$. Therefore, such $X$ cannot be a
    local minimizer.

    Now we turn to the case when $|\frac{b}{\sqrt{a^2 + b^2}}| \geq
        \frac{1}{\sqrt{2}}$. Without loss of generality, we assume that $b > 0$,
    the analysis for $b < 0$ can be carried in the same way. Again, we can
    find a direction $d$ satisfying
    \begin{equation*}
        \cgt{d}d =1, \quad \cgt{d}x_1 = 0, \quad \cgt{d}x_2 < -\frac{1}{2}, \quad \text{and} \quad \cgt{d}x_j = 0, ~j =3, \ldots, p.
    \end{equation*}
    Let $\tilde{d}=-\mi d$, then $\cgt{\tilde{d}}x_2 = \mi \cgt{d}x_2$. In
    the same manner, we construct a perturbation $\tX$ of $X$ such that
    \begin{equation*}
        \tilde{x}_1=\sqrt{1-\veps^2}x_1 + \veps \tilde{d}, \quad \tilde{x}_j = x_j, ~j=2, \ldots, p.
    \end{equation*}
    Denote $c=\cgt{d}x_2 < 0$. At this time, we have
    \begin{align*}
         & \cgt{\tilde{x}_1} \tilde{x}_1 = 1,                                                                                             \\
         & \cgt{\tilde{x}_1}\tx_2 = \sqrt{1-\veps^2}\cgt{x_1}x_2 + \veps \cgt{\tilde{d}}x_2 = \sqrt{1-\veps^2} (a + \mi b) + \mi \veps c, \\
         & \cgt{\tilde{x}_1}\tx_j = \sqrt{1-\veps^2}\cgt{x_1}x_j, \quad j=3, \ldots, p.
    \end{align*}
    Since $b=\imp(\cgt{x_1}x_2) > 0$, let $\veps > 0$ be small enough such that
    $\sqrt{1-\veps^2}b + \veps c > 0$, we have
    \begin{align*}
        |\sqrt{1-\veps^2}\cgt{x_1}x_2 + \veps \cgt{\tilde{d}}x_2| - |\cgt{x_1}x_2| \leq \frac{\veps^2c^2}{2\sqrt{a^2+b^2}} + \frac{\veps c}{2}.
    \end{align*}
    Following the same discussion we have $E_1(\tX) < E_1(X)$ with
    $\fnorm{\tX - X} \leq \veps$ and thus such $X$ cannot be the local
    minimizer of $E_1(X)$.

    We have shown that any local minimizer of qL1M satisfies the orthogonal
    constraint and hence is a local minimizer of classical trace
    minimization method,
    \begin{equation*}
        \min_{X\in\bbC^{n\times p}} \tr(\cgt{X}AX), \quad \text{s.t.} \quad \cgt{X}X=I.
    \end{equation*}
    The analysis for the trace minimization method \cite{sameh1982trace}
    tells that local minimizers take the form $X=Q_p V^*$, with
    $V\in\bbC^{p\times p}$ unitary and $Q_p$ are eigenvectors corresponding
    to the $p$ smallest eigenvalues $\Lambda_p$ of matrix $A$.
    Furthermore, by substituting $X=Q_p V^*$ into the energy functional
    $E_1(X)$ we obtain the same objective value $E_1(Q_p
        V^*)=\tr(\Lambda_p)$. Thus, any local minimum of qL1M is also a global
    minimum.

    On the other hand, note that oblique manifold $\ob(n,p)$ is a bounded
    and closed set and the objective function $E_1(X)$ is continuous with
    respect to $X$. Thus, qL1M must possess a global minimum over
    $\ob(n,p)$, which implies that any matrix in the form of $X=Q_p V^*$ is
    a global minimizer of qL1M and there are no other local minimizers.
\end{proof}

\bibliographystyle{siam}
\bibliography{qES.bib}

\begin{thebibliography}{10}

\bibitem{abrams1999quantum}
{\sc D.~S. Abrams and S.~Lloyd}, {\em Quantum algorithm providing exponential
  speed increase for finding eigenvalues and eigenvectors}, Physical Review
  Letters, 83 (1999), p.~5162.

\bibitem{arute2019quantum}
{\sc F.~Arute, K.~Arya, R.~Babbush, D.~Bacon, J.~C. Bardin, R.~Barends,
  R.~Biswas, S.~Boixo, F.~G. Brandao, D.~A. Buell, et~al.}, {\em Quantum
  supremacy using a programmable superconducting processor}, Nature, 574
  (2019), pp.~505--510.

\bibitem{aspuru2005simulated}
{\sc A.~Aspuru-Guzik, A.~D. Dutoi, P.~J. Love, and M.~Head-Gordon}, {\em
  Simulated quantum computation of molecular energies}, Science, 309 (2005),
  pp.~1704--1707.

\bibitem{bauer2020quantum}
{\sc B.~Bauer, S.~Bravyi, M.~Motta, and G.~K.-L. Chan}, {\em Quantum algorithms
  for quantum chemistry and quantum materials science}, Chemical Reviews, 120
  (2020), pp.~12685--12717.

\bibitem{bierman2022quantum}
{\sc J.~Bierman, Y.~Li, and J.~Lu}, {\em Quantum orbital minimization method
  for excited states calculation on a quantum computer}, Journal of Chemical
  Theory and Computation, 18 (2022), pp.~4674--4689.

\bibitem{blunt2015excited}
{\sc N.~Blunt, S.~D. Smart, G.~H. Booth, and A.~Alavi}, {\em An excited-state
  approach within full configuration interaction quantum {Monte} {Carlo}}, The
  Journal of Chemical Physics, 143 (2015).

\bibitem{chen2024continuity}
{\sc H.~Chen and Y.~Li}, {\em On the continuity of {Schur}-{Horn} mapping},
  arXiv preprint arXiv:2407.00701,  (2024).

\bibitem{colless2018computation}
{\sc J.~I. Colless, V.~V. Ramasesh, D.~Dahlen, M.~S. Blok, M.~E.
  Kimchi-Schwartz, J.~R. McClean, J.~Carter, W.~A. de~Jong, and I.~Siddiqi},
  {\em Computation of molecular spectra on a quantum processor with an
  error-resilient algorithm}, Physical Review X, 8 (2018), p.~011021.

\bibitem{deglmann2015application}
{\sc P.~Deglmann, A.~Sch{\"a}fer, and C.~Lennartz}, {\em Application of quantum
  calculations in the chemical industry—an overview}, International Journal
  of Quantum Chemistry, 115 (2015), pp.~107--136.

\bibitem{ganzhorn2019gate}
{\sc M.~Ganzhorn, D.~J. Egger, P.~Barkoutsos, P.~Ollitrault, G.~Salis, N.~Moll,
  M.~Roth, A.~Fuhrer, P.~Mueller, S.~Woerner, et~al.}, {\em Gate-efficient
  simulation of molecular eigenstates on a quantum computer}, Physical Review
  Applied, 11 (2019), p.~044092.

\bibitem{gao2024weighted}
{\sc W.~Gao, Y.~Li, and H.~Shen}, {\em Weighted trace-penalty minimization for
  full configuration interaction}, SIAM Journal on Scientific Computing, 46
  (2024), pp.~A179--A203.

\bibitem{grimsley2019adaptive}
{\sc H.~R. Grimsley, S.~E. Economou, E.~Barnes, and N.~J. Mayhall}, {\em An
  adaptive variational algorithm for exact molecular simulations on a quantum
  computer}, Nature communications, 10 (2019), p.~3007.

\bibitem{higgott2019variational}
{\sc O.~Higgott, D.~Wang, and S.~Brierley}, {\em Variational quantum
  computation of excited states}, Quantum, 3 (2019), p.~156.

\bibitem{holmes2017excited}
{\sc A.~A. Holmes, C.~Umrigar, and S.~Sharma}, {\em Excited states using
  semistochastic heat-bath configuration interaction}, The Journal of Chemical
  Physics, 147 (2017).

\bibitem{javadi2024quantum}
{\sc A.~Javadi-Abhari, M.~Treinish, K.~Krsulich, C.~J. Wood, J.~Lishman,
  J.~Gacon, S.~Martiel, P.~D. Nation, L.~S. Bishop, A.~W. Cross, et~al.}, {\em
  Quantum computing with {Qiskit}}, arXiv preprint arXiv:2405.08810,  (2024).

\bibitem{jones2019variational}
{\sc T.~Jones, S.~Endo, S.~McArdle, X.~Yuan, and S.~C. Benjamin}, {\em
  Variational quantum algorithms for discovering {Hamiltonian} spectra},
  Physical Review A, 99 (2019), p.~062304.

\bibitem{kandala2017hardware}
{\sc A.~Kandala, A.~Mezzacapo, K.~Temme, M.~Takita, M.~Brink, J.~M. Chow, and
  J.~M. Gambetta}, {\em Hardware-efficient variational quantum eigensolver for
  small molecules and quantum magnets}, nature, 549 (2017), pp.~242--246.

\bibitem{kitaev1995quantum}
{\sc A.~Y. Kitaev}, {\em Quantum measurements and the {Abelian} stabilizer
  problem}, arXiv preprint quant-ph/9511026,  (1995).

\bibitem{lee2018generalized}
{\sc J.~Lee, W.~J. Huggins, M.~Head-Gordon, and K.~B. Whaley}, {\em Generalized
  unitary coupled cluster wave functions for quantum computation}, Journal of
  chemical theory and computation, 15 (2018), pp.~311--324.

\bibitem{li2020optimal}
{\sc Y.~Li and J.~Lu}, {\em Optimal orbital selection for full configuration
  interaction ({OptOrbFCI}): Pursuing the basis set limit under a budget},
  Journal of Chemical Theory and Computation, 16 (2020), pp.~6207--6221.

\bibitem{li2019coord}
{\sc Y.~Li, J.~Lu, and Z.~Wang}, {\em {Coordinate-wise} descent methods for
  leading eigenvalue problem}, SIAM Journal on Scientific Computing, 41 (2019),
  pp.~A2681--A2716.

\bibitem{liu2015efficient}
{\sc X.~Liu, Z.~Wen, and Y.~Zhang}, {\em An efficient {Gauss}--{Newton}
  algorithm for symmetric low-rank product matrix approximations}, SIAM Journal
  on Optimization, 25 (2015), pp.~1571--1608.

\bibitem{LU201787}
{\sc J.~Lu and K.~Thicke}, {\em Orbital minimization method with {$\ell_1$}
  regularization}, Journal of Computational Physics, 336 (2017), pp.~87--103.

\bibitem{mauri1993orbital}
{\sc F.~Mauri, G.~Galli, and R.~Car}, {\em Orbital formulation for
  electronic-structure calculations with linear system-size scaling}, Physical
  Review B, 47 (1993), p.~9973.

\bibitem{mcardle2020quantum}
{\sc S.~McArdle, S.~Endo, A.~Aspuru-Guzik, S.~C. Benjamin, and X.~Yuan}, {\em
  Quantum computational chemistry}, Reviews of Modern Physics, 92 (2020),
  p.~015003.

\bibitem{mcclean2017hybrid}
{\sc J.~R. McClean, M.~E. Kimchi-Schwartz, J.~Carter, and W.~A. De~Jong}, {\em
  Hybrid quantum-classical hierarchy for mitigation of decoherence and
  determination of excited states}, Physical Review A, 95 (2017), p.~042308.

\bibitem{mcclean2016theory}
{\sc J.~R. McClean, J.~Romero, R.~Babbush, and A.~Aspuru-Guzik}, {\em The
  theory of variational hybrid quantum-classical algorithms}, New Journal of
  Physics, 18 (2016), p.~023023.

\bibitem{nakanishi2019subspace}
{\sc K.~M. Nakanishi, K.~Mitarai, and K.~Fujii}, {\em Subspace-search
  variational quantum eigensolver for excited states}, Physical Review
  Research, 1 (2019), p.~033062.

\bibitem{nocedal1999numerical}
{\sc J.~Nocedal and S.~J. Wright}, {\em Numerical optimization}, Springer,
  1999.

\bibitem{ollitrault2020quantum}
{\sc P.~J. Ollitrault, A.~Kandala, C.-F. Chen, P.~K. Barkoutsos, A.~Mezzacapo,
  M.~Pistoia, S.~Sheldon, S.~Woerner, J.~M. Gambetta, and I.~Tavernelli}, {\em
  Quantum equation of motion for computing molecular excitation energies on a
  noisy quantum processor}, Physical Review Research, 2 (2020), p.~043140.

\bibitem{ordejon1993unconstrained}
{\sc P.~Ordej{\'o}n, D.~A. Drabold, M.~P. Grumbach, and R.~M. Martin}, {\em
  Unconstrained minimization approach for electronic computations that scales
  linearly with system size}, Physical Review B, 48 (1993), p.~14646.

\bibitem{parrish2019quantum}
{\sc R.~M. Parrish, E.~G. Hohenstein, P.~L. McMahon, and T.~J. Mart{\'\i}nez},
  {\em Quantum computation of electronic transitions using a variational
  quantum eigensolver}, Physical review letters, 122 (2019), p.~230401.

\bibitem{peruzzo2014variational}
{\sc A.~Peruzzo, J.~McClean, P.~Shadbolt, M.-H. Yung, X.-Q. Zhou, P.~J. Love,
  A.~Aspuru-Guzik, and J.~L. O’brien}, {\em A variational eigenvalue solver
  on a photonic quantum processor}, Nature communications, 5 (2014), p.~4213.

\bibitem{preskill2018quantum}
{\sc J.~Preskill}, {\em Quantum computing in the {NISQ} era and beyond},
  Quantum, 2 (2018), p.~79.

\bibitem{romero2018strategies}
{\sc J.~Romero, R.~Babbush, J.~R. McClean, C.~Hempel, P.~J. Love, and
  A.~Aspuru-Guzik}, {\em Strategies for quantum computing molecular energies
  using the unitary coupled cluster ansatz}, Quantum Science and Technology, 4
  (2018), p.~014008.

\bibitem{ryabinkin2018qubit}
{\sc I.~G. Ryabinkin, T.-C. Yen, S.~N. Genin, and A.~F. Izmaylov}, {\em Qubit
  coupled cluster method: a systematic approach to quantum chemistry on a
  quantum computer}, Journal of chemical theory and computation, 14 (2018),
  pp.~6317--6326.

\bibitem{sameh1982trace}
{\sc A.~H. Sameh and J.~A. Wisniewski}, {\em A trace minimization algorithm for
  the generalized eigenvalue problem}, SIAM Journal on Numerical Analysis, 19
  (1982), pp.~1243--1259.

\bibitem{santagati2017finding}
{\sc R.~Santagati, J.~Wang, A.~Gentile, S.~Paesani, N.~Wiebe, J.~McClean,
  D.~Bonneau, J.~Silverstone, S.~Morley-Short, P.~Shadbolt, et~al.}, {\em
  Finding excited states of physical {Hamiltonians} on a silicon quantum
  photonic device}, in Frontiers in Optics, Optica Publishing Group, 2017,
  pp.~FM4E--2.

\bibitem{santagati2018witnessing}
{\sc R.~Santagati, J.~Wang, A.~A. Gentile, S.~Paesani, N.~Wiebe, J.~R. McClean,
  S.~Morley-Short, P.~J. Shadbolt, D.~Bonneau, J.~W. Silverstone, et~al.}, {\em
  Witnessing eigenstates for quantum simulation of {Hamiltonian} spectra},
  Science advances, 4 (2018), p.~eaap9646.

\bibitem{schriber2017adaptive}
{\sc J.~B. Schriber and F.~A. Evangelista}, {\em Adaptive configuration
  interaction for computing challenging electronic excited states with tunable
  accuracy}, Journal of chemical theory and computation, 13 (2017),
  pp.~5354--5366.

\bibitem{sun2018pyscf}
{\sc Q.~Sun, T.~C. Berkelbach, N.~S. Blunt, G.~H. Booth, S.~Guo, Z.~Li, J.~Liu,
  J.~D. McClain, E.~R. Sayfutyarova, S.~Sharma, S.~Wouters, and G.~K.-L. Chan},
  {\em {PySCF}: the {Python}-based simulations of chemistry framework}, Wiley
  Interdisciplinary Reviews: Computational Molecular Science, 8 (2018),
  p.~e1340.

\bibitem{tilly2022variational}
{\sc J.~Tilly, H.~Chen, S.~Cao, D.~Picozzi, K.~Setia, Y.~Li, E.~Grant,
  L.~Wossnig, I.~Rungger, G.~H. Booth, and J.~Tennyson}, {\em The variational
  quantum eigensolver: a review of methods and best practices}, Physics
  Reports, 986 (2022), pp.~1--128.

\bibitem{wang2019coordinate}
{\sc Z.~Wang, Y.~Li, and J.~Lu}, {\em Coordinate descent full configuration
  interaction}, Journal of chemical theory and computation, 15 (2019),
  pp.~3558--3569.

\bibitem{wen2016trace}
{\sc Z.~Wen, C.~Yang, X.~Liu, and Y.~Zhang}, {\em Trace-penalty minimization
  for large-scale eigenspace computation}, Journal of Scientific Computing, 66
  (2016), pp.~1175--1203.

\bibitem{zhong2020quantum}
{\sc H.-S. Zhong, H.~Wang, Y.-H. Deng, M.-C. Chen, L.-C. Peng, Y.-H. Luo,
  J.~Qin, D.~Wu, X.~Ding, Y.~Hu, et~al.}, {\em Quantum computational advantage
  using photons}, Science, 370 (2020), pp.~1460--1463.

\end{thebibliography}

\end{document}